\documentclass[11pt]{amsart}
\usepackage{fullpage,url,amssymb,enumerate,colonequals}
\usepackage{graphicx} 
\usepackage{amsmath}
\usepackage{amsthm}
\usepackage{amssymb}
\usepackage{tikz-cd}
\usepackage{tikz}
\usepackage{tabularx}
\usepackage{footnote}
\usepackage{hyperref}

\makesavenoteenv{table}
\makesavenoteenv{tabularx}

\DeclareMathOperator{\pr}{pr}

\DeclareMathOperator{\Sh}{Sh}
\DeclareMathOperator{\CH}{CH}
\DeclareMathOperator{\Cer}{Cer}
\DeclareMathOperator{\Pic}{Pic}
\DeclareMathOperator{\Fr}{Fr}
\newtheorem{theorem}{Theorem}[section]
\newtheorem{proposition}[theorem]{Proposition}

\newtheorem{corollary}[theorem]{Corollary}

\theoremstyle{theorem}
\newtheorem*{conjecture}{Conjecture}

\theoremstyle{definition}
\newtheorem*{remark}{Remark}

\newtheorem{definition}{Definition}
\newtheorem*{Ack}{Acknowledgements}
\newtheorem*{eg}{Example}

\newcommand{\Q}{\mathbb{Q}}

\newcommand{\tors}{{\operatorname{tors}}}

\title{Ceresa cycles of $X_{0}( N)$ }
\author{Elvira Lupoian and James Rawson}
\address{Elvira Lupoian \\ Department of Mathematics, University College London, London \\ United Kingdom}
\email{e.lupoian@ucl.ac.uk}
\address{James Rawson \\ School of Mathematics and Statistics, University of Glasgow, Glasgow\\ United Kingdom}
\email{james.rawson@glasgow.ac.uk}

\begin{document}

\begin{abstract}
The Ceresa cycle is an algebraic 1-cycle on the Jacobian of an algebraic curve. Although it is homologically trivial, Ceresa famously proved that for a very general complex curve of genus at least 3, it is non-trivial in the Griffiths group. In this paper we study the Ceresa cycle attached to the complete modular curve $X_{0}(N)$ modulo rational equivalence. For prime level $p$ we give a complete description, namely we prove that if $X_{0}(p)$ is not hyperelliptic, then its Ceresa cycle is non-torsion. For general level $N$, we prove that there are finitely many $X_{0}(N)$ with torsion Ceresa cycle. Our method relies on the relationship between the vanishing of the Ceresa cycle and Chow-Heegner points on the Jacobian. We use the geometry and arithmetic of modular Jacobians to prove that such points are of infinite order and therefore deduce non-vanishing of the Ceresa cycle.
\end{abstract}
\maketitle

\tableofcontents
\section{Introduction}
Let $C$ be a smooth, projective and geometrically integral curve, defined over an algebraically closed field $K$ of characteristic $0$. Fix a degree $1$ base point divisor $e$ and let $\iota_{e}$ be  the corresponding Abel-Jacobi map, which embeds the curve into its Jacobian $J$. The \textit{Ceresa cycle} associated to $(C,e)$ is defined as
\begin{center}
    $\Cer(C,e) := [\iota_{e}(C)] - [-1]^* [\iota_{e}(C)] $
\end{center}
which is an element of the Chow group $\CH_{1}(J)$, modulo rational equivalence.
This cycle is always homologically trivial, that is, it lies in the kernel of the cycle class map from $\CH_{1}(J)$ into a Weil cohomology theory, but crucially it is not always trivial in the Chow group. In fact, for a very general complex curve of genus at least 3, Ceresa \cite{ceresa1983c} proved that it is algebraically, and in particular rationally, non-trivial independent of the choice of base point divisor.

In the case of hyperelliptic curves, taking $e$ to be a Weierstrass point shows that $\Cer(C,e) =0$. Beyond this, there are few other explicit examples of curves where the vanishing/non-vanishing of the Ceresa cycle is known. 

\subsection{Main Results}
In this paper, we study a classically interesting family of curves, namely the modular curves whose non-cuspidal points parameterise isomorphism classes of elliptic curves with a marked isogeny. We study the Ceresa cycle of $X_{0}(N)$ inside $\CH_1(J_{0}(N)) \otimes \mathbb{Q}$, henceforth simply $\CH_1(J_{0}(N))$. As detailed in Theorem \ref{zhang} (cf. \;\cite{zhang2010gross}), the Ceresa cycle can only possibly vanish for one choice of base point divisor class, namely $e = \frac{K_{C}}{2g -2}$, where $K_{C}$ is the canonical class and $g$ is the genus of the curve. For the rest of this paper we write  $\Cer(N)$ for $\Cer\left(X_0(N), \frac{K_{X_0(N)}}{2g -2}\right)$. The first main result of our paper gives a complete classification for prime level $X_0(N)$.

\begin{theorem} \label{THM1}
Let $p$ be a prime number. The Ceresa cycle $\Cer(p)$ is non-zero if and only if $X_0(p)$ is not hyperelliptic. That is, if $p > 71$ or $p \in \{43, 53, 61, 67\}$, then $\Cer(p)$ is non-trivial.
\end{theorem}

As $\Cer(p)$ is zero whenever $X_0(p)$ is hyperelliptic, Theorem \ref{THM1} is optimal. The vanishing of the Ceresa cycle on $X_{0}(N)$ is related to certain (scaled) Chow-Heegner divisors, first introduced by Darmon, Rotger and Sols \cite{darmon2012iterated}. More specifically, for any correspondence on $X_0(N)$,  $\alpha$,  we define a rational divisor on the curve $\Sh(\alpha)$, which is referred to as a \textit{shadow point} or \textit{shadow divisor} throughout this paper. It follows (Proposition \ref{cer0shadowtors}) that if $\Cer(N)$ vanishes, then $\Sh(\beta)$ defines a rational torsion point on the Jacobian $J_0(N)$ for any correspondence $\beta$ on $X_0(N)$, and this gives us a strategy for the  proof of Theorem \ref{THM1}.

Hecke operators provide a natural source of correspondences on $X_0(N)$, and we proceed by showing that for certain Hecke operators $T$, the linear equivalence class, $D$, of the divisor $\Sh(T)$ is non-torsion in $J_0(p)(\Q)$, systematically, depending on the congruence class of $p$ modulo $12$.  For $p \equiv 5,7 \mod{12}$, we set $D$ to be the linear equivalence class of $3\Sh(T_2)$ or $6 \Sh(T_2)$, depending on the congruence classes of $ p$ modulo $8$ and $7$. We compute these divisors explicitly, in terms of CM points and cusps, using the definition of $T_2$ and the moduli interpretation of the curve. 
We prove that $D$ is non-torsion by projecting it on the Jacobian of $X_0^+(p)$, the quotient of $X_0(p)$ by the Fricke involution $w_p$. The Jacobian of this curve has no non-zero rational torsion, and we proceed by reducing $D$ modulo $p$ and by arguing that this divisor is not cut out by a function on the curve using the special fibre of $X_0^+(p)$ at $p$ and explicit $j$-invariants, see \S\ref{57mod12sec} for details. 

A similar approach is used for $p \equiv1 \mod{12}$, unless $p \equiv 1 \mod{8}$ and $p \equiv 1,2,4 \mod{7}$, in which case we show that either $6 \Sh(T_2)$ or $ 6 \Sh(T_3)$ is non-torsion by studying the relations on $j$-invariants that would be imposed if both were zero. See \S\ref{1mod12sec} for details.

In the case of $p \equiv 11 \mod{12}$, we cannot guarantee that $ \Sh(T_\ell) \ne 0 $ for any fixed $\ell$. Instead,  we work $  \Sh(w_p)$, where $w_p$ is the Fricke involution on $X_0(p)$, and we apply Proposition \ref{M3}
(see \cite{mazur1977modular}), which controls the degrees of certain functions on $X_0^+(p)$, to deduce that for all $p \ge 104,711$, $\Sh(w_p)$ is non-torsion in the Jacobian. We deal with the remaining cases by finding an explicit $\ell$ for which $\Sh(T_\ell) \ne 0$, and we apply our previous reduction method to this point. 

Our proof makes explicit use of the arithmetic of the $J_0(p)$, its rational points $J_0(p)(\Q)$ and of the ground breaking work of Mazur \cite{mazur1977modular}. Although the divisors studied in this work are supported on CM points on the curve, they are \textbf{not} ordinary Heegner points, and one cannot simply apply the Gross-Zagier theorem \cite{gross1986heegner}, to conclude that they are non-torsion. To the best of our knowledge, this is the first work in which reduction (at bad primes) and related arithmetic tricks have been used to prove certain points on $J_0(p)(\Q)$ are systematically non-torsion.  

Our second main result deals with the case of arbitrary level $N$ and has significantly more exclusions. 
\begin{theorem} \label{THM2}
    There are finitely many $N$ such that $\Cer(N) = 0$. More precisely, if
    \begin{center}
    $N > 2^5 \times 3^4 \times 5^2 \times 7^2 \times \displaystyle \prod_{\substack{11 \le p \le 71, \\ p \not \in \{43,53, 61,67\} }} p$
    \end{center}
    then $\Cer(N)$ is non-zero in $\CH_{1}(J_{0}(N))$.
\end{theorem}

Combining Proposition \ref{covers} (which describes how Ceresa cycles behave in covers) and Theorem \ref{THM1}, the proof of Theorem \ref{THM2}, reduces to showing that for any prime 
\begin{center}
$p \in \{2,3,5,7,11,13,17,19,23,29,31,37,41,47,59, 71 \} $
\end{center}
$\Cer(p^n)$ is non-torsion for some $n \ge 2$. We deal with these finitely many computations explicitly. More precisely, we compute explicit shadow divisors, and show that they are non-torsion by applying the Gross-Zagier formula in some special cases or by explicit period lattice computations. These computations are explained in \S\ref{proofof2}. 

While it is conceivably possible to use the methods of Section \S\ref{proofof2} to treat all levels $N$ for which $X_{0}(N)$ is non-hyperelliptic and $J_{0}(N)(\Q)$ has positive rank, there are sufficiently many cases to consider that it is impractical and too computationally expensive at this time.

One drawback of our method is that it does not apply to examples where $J_{0}(N)(\Q)$ has rank $0$. A famous example of this is $X_{0}(64)$, which is simply the Fermat quartic, whose Jacobian has $\Q$-rank zero, but for which the Ceresa cycle has infinite order \cite{harris1983homological}, \cite{bloch1984algebraic}. There are other examples excluded by the bound of Theorem \ref{THM2}, for instance $N = 34, 38 , 51 , 55$, are all examples where $J_0(N)(\Q)$ has rank $0$ and $X_0(N)$ is non-hyperelliptic.  The curve $X_{0}(34)$ has genus $3$ and we were able to verify non-vanishing using \cite{ellenberg2024certifyingnontrivialityceresaclasses}. We were unable to apply this algorithm to $N = 38, 51, 55$  due to the hypothesis that the Sato-Tate group is the whole of $\mathrm{GSp}$. Moreover, the above examples fail the cohomological algebraic vanishing criteria proved in \cite{laga2024vanishing}. At present, we are unaware of any further results on the (non)-vanishing of Ceresa cycles for any other levels below the bound of Theorem 2, or of any methods which could treat the cases where $J_0(N)(\Q)$ is of rank zero.

Beyond constructing new points of infinite order on the modular Jacobian, the main arithmetic application of our work is towards triple product $L$-functions. Let $\Gamma^{3} \in \CH_{1}(X_{0}(N)^{3})$ be the modified diagonal cycle  (an object closely related to the Ceresa cycle, see e.g. \cite{zhang2010gross}), with base divisor $\frac{K_{C}}{2g -2}$, and let $P = \CH_1(X_0(N)^3)_\Q^0/\ker^{BB}$, where $\ker^{BB}$ is the kernel of the Beilinson-Bloch pairing. For any triple $(f,g,h)$ of weight $2$ cusp forms of level $\Gamma_0(N)$ with global root number $-1$, we have a well-defined $f \times g \times h$-isotypic component $\Gamma_{fgh}^3 \in P_{\mathbb{R}}$ of the modified diagonal cycle $ \Gamma^3 \in  \CH_1(X_0(N)^3)$, see \cite[Section 13]{gross1992heights} for precise definitions.  
There is a conjectured relation between the Beilinson-Bloch height of $\Gamma_{fgh}^3 \in P_{\mathbb{R}}$ and the special value $L'(f\times g \times h, 2)$, by analogy with the Gross-Zagier height formula \cite{gross1986heegner} \cite{yuanzhangzhang} \cite{Bertolini_Castella_Darmon_Dasgupta_Prasanna_Rotger_2014}. In the last section we prove that for bielliptic $X_0(p)$ with $p$ prime, the vanishing of $\Gamma_{fgh}$ is as predicted by the triple product $L$-function. 

\begin{theorem} \label{thm3}
 Let $p \in \{ 43, 53, 61, 79, 83,  101, 131\}$ and $(f,g)$ weight $2$ cusps forms of level $\Gamma_0(p)$, with global root number $-1$. Then, $\Gamma_{ffg}$ vanishes if and only if $L'(f \times f \times g, 2) =0$. 
\end{theorem}

This is proved in  Propositions \ref{prop2} and \ref{prop1}.

\subsection{Future Directions}

\subsubsection{Modular curves and Shimura curves}
As previously mentioned, covers of curves with non-vanishing Ceresa cycles have non-vanishing Ceresa cycles, with the appropriate base point (see Proposition \ref{covers}), and therefore \textbf{any modular curve} lying over a curve $X_0(N)$ classified by Theorem \ref{THM1} or \ref{THM2}, has non-trivial Ceresa cycle. This includes the well-studied families of modular curves $X_1(N)$ and $X(N)$. Moreover, our work completely recovers the main result of the recent work of Kerr-Li-Qiu-Yang \cite{kerr2024non}, who study the complete modular curve attached to $\Gamma(2) \cap \Gamma_1(2N)$, using a different method to the one of the present paper, see \S\ref{lit} for a comment on their method.  

Theorem \ref{THM2} and Proposition \ref{covers} prove a significant part of \cite[Conjecture 1.2.2]{qiu2023finiteness}, which conjectures that there exist finitely many Shimura curves with vanishing Ceresa cycle. Our method is a natural stepping stone towards the resolution of this conjecture, with shadow points of Hecke operators being explicitly computable via the moduli interpretation (see \cite[Section 8.5]{assaf2026hyperelliptic} for the computation of some shadow divisors on Shimura curves) and reduction data is available in many cases.

For specific families of modular and Shimura curves, we expect our result to have more direct applications.
Pushing forward and pulling back the modified diagonal cycle from $X_0(N)$ can be used to deduce non-vanishing for some other families of modular curves. In particular, we plan to explore similar applications of Theorems \ref{THM1} and \ref{THM2} to the families $X_s^+(p)$ and $X_0^+(p)$ in future work. 

Moreover, there are other natural algebraic cycles on modular curves, see for instance \cite{lilienfeldt2026twisted}, to which one might hope our method generalises.
\subsubsection{New method for verifying non-vanishing} 
Away from modular and Shimura curves, our method naturally extends to verifying non-vanishing of the Ceresa cycle for curves $C/\Q$ with Jacobian $J$ for which we have the following:
\begin{itemize}
    \item  a non-trivial correspondence of $C$ over $\Q$;
    \item  an upper bound for the order of the rational torsion subgroup of $J$;
    \item reduction data at a prime of bad reduction $p >2$, such as a regular (not necessarily minimal) model of $C$.
\end{itemize}
In specific examples, shadow points can be computed as divisors on the curve and the reduction at $p$, alongside the torsion information, can be used to deduce that such a divisor is non-torsion in the Jacobian. We believe it is possible to automate this method, and for the resulting algorithm to be implementable and efficient to run for curves of small genus. Notably, such an algorithm would complement the recent work of Ellenberg, Logan and Srinivasan \cite{ellenberg2024certifyingnontrivialityceresaclasses}. 

In practice, all of the above data is accessible given a curve with a correspondence. A bound for the torsion subgroup of $J(\Q)$ can be obtained by reducing modulo multiple primes, and using the fact that reduction modulo an odd prime is injective on torsion \cite{katz}. Regular models can be computed by blowing up any model of the curve. 

In addition to providing an algorithmic approach for individual curves, there are other families of curves to which these hypotheses apply. For example, the hypergeometric curves, $(x^N - 1)(y^N - 1) = \lambda$ for some parameters $N \in \mathbb{N}$ and $\lambda \in \mathbb{C} \setminus \{0, 1\}$, considered in \cite{EskandariNemoto}, where the shadow method was applied but they were unable to determine if the shadow divisors were torsion. One might hope that, at least for suitable choices of the parameters $N$ and $\lambda$, these techniques will apply there too.

\subsubsection{Explicit points on $J_0(N)(\Q)$}

An interesting by-product of our work is the description of explicit points of infinite order in $J_0(p)(\Q)$, and subsequently on $J_0^+(p)(\Q)$. In Table \ref{table:2} we give the explicit description of the divisors for $p \not \equiv 11 \mod{12}$. Notably, these points are supported on cusps and CM points, but they are not ordinary Heegner points. We expect these points to have wider arithmetic applications. 

Beyond their role in the present work, shadow divisors and fixed point divisors of Hecke correspondences appear to be of greater arithmetic significance. In particular, they have been studied in the context of the quadratic Chabauty method \cite{dogra2021quadratic}, \cite{hashimoto2026equationless}. Our work gives a method for working directly with these points.

\subsubsection{$L$-functions, Gross--Kudla conjecture and Beilinson--Bloch conjecture}
As previously alluded to, we foresee one of the main application of our work to be towards verifying non-vanishing of the Beilinson-Bloch height of the  projections of the modified diagonal cycle in the conjectured Gross-Kudla formula. More precisely, consider any three weight $2$ cusps forms $f,g,h$ of level $\Gamma_{0}(N)$, with $N$ squarefree, and such that the $L$-function of  $F = f \times g \times h $  has global root number $-1$.  Gross and Kudla \cite{gross1992heights} predict that 
\begin{center}
 $L'(F,2) = \Omega(F) \langle \Gamma^{3}_{F}, \Gamma^{3}_{F} \rangle_{BB} $   
 \end{center}
 where $\Omega(F)$ is an explicit,non-zero constant, $\Gamma^{3}_{F}$ is the $F$-isotypic component of $\Gamma^{3}$ and $\langle, \rangle_{BB}$ is the Bloch-Beilinson height pairing.

In explicit examples, showing vanishing or non-vanishing of  $\langle \Gamma^{3}_{F}, \Gamma^{3}_{F} \rangle_{BB} \ne 0$ can be achieved in a similar manner to Theorem \ref{thm3}, by exploiting the arithmetic of the curve, and by-passing any difficult height calculations. In particular, testing the hypothesis that the $L$-derivative vanishes if and only if the cycle vanishes can be done on a larger scale using our methods, and the strategy outlined in Section 7.

\subsection{Related Works} \label{lit}

The recent work of Kerr, Li, Qiu and Yang \cite{kerr2024non} proved that the Ceresa cycle of  carefully constructed modular curves (those attached to the congruence subgroup $\Gamma(2) \cap \Gamma_1(2N)$) is non-vanishing for $N$ sufficiently large. This work also relies on the well-established relation between Ceresa cycles, modified diagonal cycles and the shadow divisors (which are Chow--Heegner divisors in their language), but the treatment of shadow divisors differs significantly from the present paper. More precisely, that work finds linear combinations of Hecke operators such that the corresponding shadow points are genuine Heegner points, and can then conclude using the Gross-Zagier theorem \cite{gross1986heegner}. This approach is restrictive, and appears to be tailor-made for modular curves of very specific levels.
Theorem \ref{THM2} completely recovers the main result of \cite{kerr2024non}, and one expects our method to be applicable to more families of modular curves as well as other curves.  

Some examples of Shimura curves with zero Ceresa cycle were discovered by Qiu-Zhang \cite{qiu2023vanishingresultsmodifieddiagonal}, \cite{qiu2023finiteness}. These works show that the modified diagonal cycle vanishes by relating it to its projection into the degree $1$ part of the Chow motive of the curve. The Chow motive is decomposed with respect to the decomposition of the Jacobian into simple abelian varieties, and the diagonal cycle is projected onto these pieces. For Shimura curves, the pieces correspond to automorphic representations and vanishing is translated into an automorphic criterion.

In \cite{ellenberg2024certifyingnontrivialityceresaclasses}, Ellenberg, Logan and Srinivasan developed an algorithm for verifying non-vanishing of the Ceresa cycle, which is guaranteed to terminate if the Sato-Tate group is $\mathrm{GSp}$, and thus the curves for which this works definitively are ``reasonably generic''. As mentioned in the previous subsection, we expect our method can be used to verify non-vanishing for curves with extra correspondences and a good theory of reduction. These two approaches synergise well, and treat many curves.

Besides the above mentioned works, various families of curves with vanishing or non-vanishing Ceresa cycles have been studied, for both rational and algebraic equivalence. The Ceresa cycle has been studied for Fermat curves and their quotients \cite{eskandari2021harmonic}, \cite{kimura2000modified}, \cite{tadokoro2016harmonic}, \cite{otsubo2012abel}, \cite{lilienfeldt2023experiments}, \cite{bloch1984algebraic}, \cite{harris1983homological}. Bielliptic Picard curves were studied in \cite{laga2023ceresa} and \cite{laga2024vanishing}. Examples of genus $3$ curves with vanishing Ceresa cycle were given by Buhler, Schoen and Top \cite{buhler1997cycles}.

\subsection{Notation and Conventions}
Throughout this paper we use the following notation and conventions. 
\begin{itemize}
    \item $N$ : a positive integer;
    \item $X_0(N)$ : the modular curve parametrising elliptic curves with a cyclic $N$-torsion subgroup;
    \item $J_0(N)$ :  the Jacobian of $X_0(N)$;
    \item $K_{X_0(N)}$ :  a choice of canonical divisor on $X_0(N)$;
    \item $\Cer(N)$ :  the Ceresa cycle on $X_0(N)$ with base divisor $\frac{K_{X_0(N)}}{2g -2}$, where $g$ is the genus of $X_0(N)$;
    \item $w_{d}$ : the Atkin-Lehner involution on $X_0(N)$ defined by $d \vert N$ with $\gcd(d, \frac{N}{d}) =1$;
    \item $T_{\ell}$ : the Hecke correspondence for a prime $\ell$;
    \item $F_{\phi}$ : the fixed point divisor of a correspondence $\phi$ on $X_0(N)$;
    \item $\Sh(\phi)$ : the shadow divisor of a correspondence $\phi$ on $X_0(N)$;
    \end{itemize}

For an integer $d \ge 1$, we write $D_d$ for the effective divisor on $X_0(N)$ supported on CM points corresponding to an order of discriminant $-d$ and multiplicity $1$. We set $D_d$ to be the zero divisor if no such points exist on $X_0(N)$. 

For any variety $X$ we denote by $\CH_1(X)$ the first Chow group, modulo rational equivalence, and tensored with $\Q$.

\subsection{Outline}
This paper is organised as follows. In Section \ref{CHpoints}, we recall some standard constructions associated to Chow groups and define `shadow points' of  correspondences on curves.  
In Section \ref{modularcurves}, we give a brief overview of some well-known results on the geometry and arithmetic of modular curves which allow us to study the shadow points. In Section \ref{heckops} we give explicit formulae for shadow points corresponding to Hecke operators on $J_{0}(N)$.
In Sections \ref{proofof1} and \ref{proofof2}, we prove Theorems \ref{THM1} and \ref{THM2} respectively. In Section \ref{heckestructure} we discuss the Hecke structure of shadow points and the links between our work and the Gross-Kudla conjecture. 

Some cases in our analysis required the use of computer calculation. The source code for these is available at 
\url{https://github.com/jameswrawson/ModularCeresa}.

\begin{Ack}
The authors thank Vladimir Dokchitser, Jef Laga, Samir Siksek and Jan Vonk for helpful conversations; and Benedict Gross for his comments on the previous version of this work which have inspired Section 7. We thank the anonymous referee for their careful reading of this paper; their suggestions have improved this paper greatly.  Both authors are grateful for the financial support provided by the UK Engineering and Physical Sciences Research Council. The first named author is supported by the EPSRC Doctoral Prize fellowship EP/W524335/1. The second named author is supported by the EPSRC studentship EP/W523793/1.\end{Ack}

\section{Gross-Kudla-Schoen Cycles and Chow-Heegner Points}
\label{CHpoints}
In what follows $C$ and $D$ are smooth, projective and geometrically integral curves, defined over an algebraically closed field $K$. We write $J_{C}$ for the Jacobian variety of $C$ and we denote the genus of $C$ by $g$.
\subsection{The Ceresa Cycle and Gross-Kudla-Schoen Classes}
Another cycle that can be written down for an arbitrary curve $C$ and degree 1 divisor $e$ on $C$ is the following.
\begin{definition}
    The \textit{Gross-Kudla-Schoen (modified) diagonal cycle}, $\Gamma^3(C, e) \in \CH_1(C^3)$, is defined to be 
    $$\Delta^3_C - \pr_{1, 2}^*(\Delta_C) \cdot \pr^*_3 e - \pr^*_{1, 3} (\Delta_C) \cdot \pr_2^* e - \pr^*_{2, 3}(\Delta_C) \cdot \pr_1^* e + \pr_{1, 2}^* (e \times e) + \pr_{1, 3}^* (e \times e) + \pr_{2, 3}^* (e \times e),$$
    where $\pr_i$ denotes the projection onto the $i$th factor, $\pr_{i, j}$ is the projection onto the product of the $i$th and $j$th factor, $\Delta_C^{3}$ denotes the small diagonal in $C^3$ (i.e. the locus where all 3 points are equal), and $\Delta_C$ the diagonal in $C^2$.
\end{definition}

The following theorem of Zhang \cite[Theorem 1.5.5]{zhang2010gross} relates the vanishing of the Ceresa cycle to the vanishing of the Gross-Kudla-Schoen diagonal cycle.
\begin{theorem}[Zhang] 
If $\Gamma^3(C, e) = 0$, then $(2g - 2)e = K_C \in \CH^1(C)$, and $\Cer(C, e) = 0$. Conversely, if $\Cer(C, e) = 0$, then $\Gamma^3(C, e) = 0$.
\label{zhang}
\end{theorem}

In light of this result, we restrict to the case $e = \frac{K_C}{2g - 2}$ and write $\Cer(C)$ for $\Cer(C, \frac{K_C}{2g - 2})$, for any choice of representative of the canonical class. 

\subsection{Chow-Heegner and Shadow Points}
To any correspondence of curves, $C \leftarrow X \rightarrow C$ we associate a \textit{fixed point divisor}, defined as the pushforward of the intersection of the correspondence with the diagonal under projection onto one of the factors.

\begin{definition}
Let $C$ be a curve of genus $g$, 
and $\phi$ be a correspondence from $C$ to itself. The \textit{shadow} of $\phi$ is the degree $0$ divisor on $C$ 
\begin{center}
    $\Sh(\phi) = (2g - 2)F_{\phi} - \deg(F_{\phi}) K_C - \phi(K_C) - \phi^\vee(K_C) + (\deg(\phi) + \deg(\phi^\vee))K_C$
\end{center}
where $F_{\phi}$ is the fixed point divisor of $\phi$, $K_{C}$ is a canonical divisor on $C$ and $\phi^{\vee}$ denotes the correspondence traversed in the opposite direction. 
\end{definition}

We remark that this divisor is a scaling of the corresponding Chow-Heegner point, see \cite{darmon2012iterated}, \cite{darmon2015algorithms} for details on the construction of such points. The terminology of ``shadows'' is adopted from \cite{ellenberg2024certifyingnontrivialityceresaclasses}, both due to the similarity of the construction and to avoid any notational ambiguity with the Chow group. If $\phi$ is defined over $\mathbb{Q}$, then so is $\Sh(\phi)$. The following proposition is similar to those found in \cite{ellenberg2024certifyingnontrivialityceresaclasses} and \cite{kerr2024non}.
\begin{proposition} \label{cer0shadowtors}
Suppose that $\phi$ is a non-trivial correspondence over $C$ (also defined over $K$). If the Ceresa cycle is zero in $\CH_1(J_C)$, then $[\Sh(\phi)] \in J_{C}(K)$ is torsion.
\end{proposition}
\begin{proof}
Associated to $\phi$ there is a natural class $X_{\phi}$ in $C^2$, given by the image of the correspondence. We can push this class forward through $\Gamma^3(C, e)$, in the following way $\pr_{3, *}\left(\Gamma^3(C, e) \cdot \pr_{1, 2}^* X_{\phi}\right)$, where $e = \frac{K_{C}}{2g -2}$. This can be calculated by computing this pushforward for each term in the definition of $\Gamma^3(C, e)$. We summarise these calculations in Table~\ref{tab:corres} using the notation of the previous definition.

\begin{table}[ht]
    \begin{tabular}{|c|c|c|}
        \hline
        $z \in \CH_1(C^3)$ & $z \cdot \pr_{1, 2}^* X_{\phi}$ & $\pr_{3, *}(z \cdot \pr_{1, 2}^* X_{\phi})$ \\ \hline \hline
        $\Delta^3_C$ & $\Delta^3_C(F_{\phi})$ & $F_{\phi}$ \\ \hline
        $-\Delta_C \times e$ & $-\Delta^2_C(F_{\phi}) \times e$ & $-\deg(F_{\phi}) e$ \\ \hline
        $-\pr_{1, 3}^* \Delta_C \cdot \pr_{2}^* e$ & $-\pr_{1, 3}^* \Delta_C(\phi^{\vee}(e)) \cdot \pr_{2}^* e$ & $-\phi^{\vee}(e)$ \\ \hline
        $-e \times \Delta_C$ & $-e \times \Delta_C(\phi(e))$ & $-\phi(e)$ \\ \hline
        $C \times e \times e$ & $\phi^{\vee}(e) \times e \times e$ & $\deg(\phi^{\vee}) e$ \\ \hline
        $e \times C \times e$ & $e \times \phi(e) \times e$ & $\deg(\phi) e$ \\ \hline
        $e \times e \times C$ & $0$ & $0$ \\ \hline
    \end{tabular}
    \caption{Components of the calculation of $\Gamma^3(C, e)_* X_{\phi}$}
    \label{tab:corres}
\end{table}

Combining these terms, we get the following point in $J_C(K) \otimes \mathbb{Q}$
\begin{center}
$F_{\phi} - \deg(F_{\phi}) e - \phi(e) - \phi^\vee(e) + \deg(\phi) e + \deg(\phi^\vee) e$
\end{center}
which is commonly the definition of a Chow-Heegner point. Multiplying by $2g - 2$ to clear denominators gives $\Sh(\phi) = (2g - 2)F_{\phi} - \deg(F_{\phi}) K_C - \phi(K_C) - \phi^\vee(K_C) + (\deg(\phi) + \deg(\phi^\vee))K_C$.
If the Ceresa cycle is zero in the rational Chow group, then so is $\Gamma^3(C, e)$ by Theorem~\ref{zhang}, and so pushing forward $X_{\phi}$ through $\Gamma^3(C, e)$ must also be zero. This shows the resulting divisor is trivial in the Jacobian tensored with $\mathbb{Q}$, and so it is torsion.
\end{proof}

We remark that the above proposition does apply to $\phi$ being the identity correspondence. In this case, $F_{\phi}$ is $-K_C$, and the corresponding shadow is $(2g - 2)e - K_C$, which is trivial by assumption.

\begin{remark}
    The divisor class of a shadow point only depends on the endomorphism of the Jacobian it induces, rather than just the equivalence class of the correspondence. If $\phi$ and $\phi'$ are two different correspondences which induce the same endomorphism on the Jacobian, their graphs $X_{\phi}$ and $X_{\phi'}$ differ in the Chow group by horizontal and vertical components. Pushing a horizontal or vertical class forward through $\Gamma^3(C, e)$ gives zero. 
\end{remark}

\subsection{Action of the Endomorphism Algebra} \label{endact}
By assumption, $C$ admits a non-trivial correspondence. It is natural to ask how the action of these on the diagonal cycle affect the shadows. We therefore expand the definition of shadow points as follows.
\begin{definition}
    Let $Z \in \CH_1(C^3) \otimes \mathbb{C}$ and let $\phi$ be a correspondence from $C$ to $C$. Define the shadow $ \Sh(Z, \phi)$ of $Z$ under $\phi$  to 
    be 
    \begin{center}
  $(2g - 2) Z_* X_{\phi} \in J_C(K) \otimes \mathbb{C}$
    \end{center}
  where as before, $X_{\phi}$ is the graph of $\phi$ in $C \times C$. 
\end{definition}

\begin{proposition} \label{breakups}
    Fix a correspondence $\phi$ from $C$ to $C$, and take another (not necessarily distinct) correspondence $\psi$, also from $C$ to $C$. Then the shadow points satisfy the following.
    \begin{itemize}
        \item $\Sh((\psi, \mathrm{id}, \mathrm{id})_* Z, \phi) = \Sh(Z, \phi \circ \psi)$
        \item $\Sh((\mathrm{id}, \psi, \mathrm{id})_* Z, \phi) = \Sh(Z, \psi^\vee \circ \phi)$
        \item $\Sh((\mathrm{id}, \mathrm{id}, \psi)_* Z, \phi) = \psi_* \Sh(Z, \phi)$
    \end{itemize}
\end{proposition}
\begin{proof}
For the first equality we note that  
    $$((\psi, \mathrm{id}, \mathrm{id})_* Z)_* X_{\phi} = \pr_{3, *} \left( (\psi, \mathrm{id}, \mathrm{id})_* Z \cdot \pr_{1, 2}^* X_{\phi}\right).$$
    Writing $\psi$ as the composition $\pi_{2, *} \pi_1^*$, where $\pi_1 : X \to C$ and $\pi_2 : X \to C$, we get the following
    $$((\psi, \mathrm{id}, \mathrm{id})_* Z)_* X_{\phi} = \pr_{3, *} \left( (\pi_2, \mathrm{id}, \mathrm{id})_* (\pi_1, \mathrm{id}, \mathrm{id})^* Z \cdot \pr_{1, 2}^* X_{\phi}\right).$$
    Applying the projection formula \cite[Tag 0B0H]{stacks-project} to $(\pi_2, \mathrm{id}, \mathrm{id})_*$
   \begin{align*}
     ((\psi, \mathrm{id}, \mathrm{id})_* Z)_* X_{\phi} & = \pr_{3, *} (\pi_2, \mathrm{id}, \mathrm{id})_* \left((\pi_1, \mathrm{id}, \mathrm{id})^* Z \cdot (\pi_2, \mathrm{id}, \mathrm{id})^* \pr_{1, 2}^* X_{\phi}\right) \\
     & = \pr_{3, *} \left((\pi_1, \mathrm{id}, \mathrm{id})^* Z  \cdot (\pi_2, \mathrm{id}, \mathrm{id})^* \pr_{1, 2}^* X_{\phi}\right).
   \end{align*}
    By the same process in reverse
  \begin{align*}
    \pr_{3, *} \left((\pi_1, \mathrm{id}, \mathrm{id})^* Z \cdot (\pi_2, \mathrm{id}, \mathrm{id})^* \pr_{1, 2}^* X_{\phi}\right)  & =    \pr_{3, *} (\pi_1, \mathrm{id}, \mathrm{id})_* \left((\pi_1, \mathrm{id}, \mathrm{id})^* Z \cdot (\pi_2, \mathrm{id}, \mathrm{id})^* \pr_{1, 2}^* X_\phi\right) \\
    & = \pr_{3, *} \left(Z \cdot (\psi^\vee, \mathrm{id}, \mathrm{id})_* \pr_{1, 2}^* X_{\phi}\right).
    \end{align*}
    
    As the action of $(\psi^\vee, \mathrm{id}, \mathrm{id})$ and $\pr_{1, 2}$ form a Cartesian square, this can be simplified to $Z_* \left((\psi^\vee, \mathrm{id})_* X_{\phi}\right)$, and this last class is then $X_{\phi \circ \psi}$. The argument for the second is nearly identical. 

    For the third equality 
    $$((\mathrm{id}, \mathrm{id}, \psi)_* Z)_* X_{\phi}) = \pr_{3, *}\left((\mathrm{id}, \mathrm{id}, \pi_2)_* (\mathrm{id}, \mathrm{id}, \pi_1)^* Z \cdot \pr_{1, 2}^* X_{\phi}\right).$$
    Using the projection formula, and noting $\pr_{1, 2} \circ (\mathrm{id}, \mathrm{id}, \pi_2) = \pr_{1, 2}$ gives
    $$\pr_{3, *}\left((\mathrm{id}, \mathrm{id}, \pi_2)_* (\mathrm{id}, \mathrm{id}, \pi_1)^* Z  \cdot \pr_{1, 2}^* X_{\phi}\right) = \pr_{3, *} (\mathrm{id}, \mathrm{id}, \pi_2)_* \left((\mathrm{id}, \mathrm{id}, \pi_1)^* Z \cdot \pr_{1, 2}^* X_{\phi}\right).$$
    Similarly, this is equal to 
    $$\pr_{3, *} (\mathrm{id}, \mathrm{id}, \pi_2)_* (\mathrm{id}, \mathrm{id}, \pi_1)^* (Z \cdot \pr_{1, 2}^* X_{\phi}) = \pr_{3, *} (\mathrm{id}, \mathrm{id}, \psi)_* (Z \cdot \pr_{1, 2}^* X_{\phi}).$$
    As $(\mathrm{id}, \mathrm{id}, \psi)$ and $\pr_3$ give rise to a Cartesian square, we get $\psi_* \pr_{3, *} (Z \cdot \pr_{1, 2}^* X_{\phi})$.
\end{proof}

These formulae are especially interesting when $\psi$ belongs to a commutative subalgebra of the full endomorphism ring, and $Z$ is taken as an eigencomponent of $\Gamma^3(C, e)$ for this sub-algebra, as this will give equalities between certain shadows (see Section~\ref{heckestructure}). 

\section{The Arithmetic and Geometry of Modular Curves}
\label{modularcurves}
\subsection{Torsion Subgroups of Modular Jacobians}
Throughout this section $p \ge 5$ denotes a prime number. The proof of Theorem \ref{THM1} focuses on proving that the shadow $D$ of some Hecke operator on $J_{0}(p)$ is non-torsion. The rational torsion subgroup of $J_{0}( p )$ is completely understood due to the classical works of Mazur and Ogg. In this section we give an overview of some well known results. Note that $X_{0}(p)$ has two cusps, both rational, which we denote by $c_{0}$ and $c_{\infty}$.

We first recall the following result due to Ogg \cite{ogg1972rational}.
\begin{theorem}[Ogg]
 The linear equivalence class of $c_{0} - c_{\infty}$ in $J_{0} ( p)$ has order $\frac{p-1}{\text{gcd} \left( p-1, 12 \right)} $.
\end{theorem}

In the same work, Ogg conjectured that the linear equivalence of this divisor generates the entire rational torsion subgroup of $J_{0}( p )$. This was proved by Mazur using the theory of Eisenstein ideals \cite[Page 142]{mazur1977modular}.

\begin{theorem}[Mazur, Theorem III.1.2] \label{M1}
  $J_{0}( p) \left( \mathbb{Q} \right)_{\text{tors}} = \langle [ c_{0} - c_{\infty} ] \rangle.$  
\end{theorem}

To prove that $D \in J_{0}(p)\left( \mathbb{Q} \right)$ is non-torsion, we further project to the quotient of $J_{0} \left( p\right)$ by the Atkin-Lehner involution $w_{p}$, that is the involution on $J_{0}(p)$ induced by the Atkin-Lehner involution on the curve. We write $X_{0}^{+}( p)$ for the quotient curve and $J_{0}^{+}(p)$ for the quotient of the Jacobian. In the same reference, Mazur proved the following. 

\begin{theorem}[Mazur, Proposition III.2.6] \label{M2}
  If $X_{0}^{+} ( p)$ has positive genus, that is $p >73$ or $p = 37,43, 53, 61, 67 $, then $J_{0}^{+}( p ) \left( \mathbb{Q} \right)_{\text{tors}}$ is trivial. 
\end{theorem}

The proof makes use of the following results, which also appear in some of our arguments. 
 \begin{proposition}[Mazur, Lemma III.2.6] \label{M3}
 Let $d$ be an integer. If $H^{0} \left( X_{0}^{+}\left( p\right), \mathcal{O}\left( dc_{\infty} \right) \right)$ is non-trivial, then $d > \frac{p}{96}$.
\end{proposition}
 
\begin{proposition}[Mazur, Proposition III.2.8] \label{classbound} \label{M4}
 Let $\mathcal{O}$ be an order in a quadratic imaginary field. Then 
 \begin{center}
     $h \left( \mathcal{O} \right) \le \frac{1}{3} \vert \Delta \vert^{1/2} \log \vert \Delta \vert $
 \end{center}
 where $\Delta$ is the discriminant of $\mathcal{O}$.
\end{proposition}

\subsection{Canonical Divisors}
As the canonical embedding is used throughout our study of Ceresa cycles, we determine expressions for the canonical divisor of modular curves. 

Recall that for $d \ge 1$, we write $D_d$ for the effective divisor on $X_0(N)$ supported on CM points corresponding to an order of discriminant $-d$ and multiplicity $1$. We set $D_d$ to be the zero divisor if no such points exist on $X_0(N)$. 

\begin{proposition}
    The canonical divisor of the modular curve $X_0(p)$ is given as follows.
    \begin{itemize}
        \item If $p \equiv 11 \mod{12}$, $K_{X_0(p)} = \frac{p - 11}{12}(c_0 + c_{\infty})$
        \item If $p \equiv 7 \mod{12}$, $3K_{X_0(p)} = \frac{p - 11}{4}(c_0 + c_{\infty}) - 2D_3$
        \item If $p \equiv 5 \mod{12}$, $2K_{X_0(p)} = \frac{p - 11}{6}(c_0 + c_{\infty}) - D_4$
        \item If $p \equiv 1 \mod{12}$, $6K_{X_0(p)} = \frac{p - 11}{2}(c_0 + c_{\infty}) - 4D_3 - 3D_4$
    \end{itemize}
\end{proposition}
\begin{proof}
    Consider the $j$-map, $j : X_0(p) \to \mathbb{P}^1$. By the Riemann-Hurwitz formula, $K_{X_0(p)} = j^*(K_{\mathbb{P}^1}) + R$, where $R$ is the ramification divisor. The canonical divisor of $\mathbb{P}^1$ is minus twice a point, and we can take that point to be $\infty$. The ramification divisor is $(p - 1)c_{\infty}$, along with the points with $j = 0$ (except where the non-trivial automorphism stabilises the $p$-subgroup) with multiplicity 2, and the points with $j = 1728$ with the same proviso and multiplicity 1. This shows that the canonical divisor on $X_0(p)$ is 
    $$-2p c_{\infty} - 2c_0 + (p - 1)c_{\infty} + 2\{j = 0\} - 2D_{3} + \{j = 1728\} - D_{4}.$$ 
    Since $3\{j = 0\} - 2D_3$ is the vanishing of $j$, this divisor is linearly equivalent to $pc_{\infty} + c_{0}$, and similarly $2\{j = 1728\} - D_4 \sim pc_{\infty} + c_0$. We deduce 
    $$6K_{X_0(p)} = -6(p + 1)c_{\infty} - 12 c_0 + 4p c_{\infty} + 4 c_{0} - 4D_{3} + 3p c_{\infty} + 3c_0 - 3 D_{4} = (p - 6)c_{\infty} - 5c_{0} - 4D_{3} - 3 D_{4}.$$
    Using $\frac{p - 1}{2} c_{\infty} \sim \frac{p - 1}{2} c_0$, this takes the more symmetric form 
    $$6 K_{X_0(p)} = \frac{p - 11}{2} (c_{\infty} + c_0) - 4 D_3 - 3 D_4.$$

    If $p \equiv 11 \mod{12}$, then $D_3 = D_4 = 0$. In this case $6K_{X_0(p)} = \frac{p - 11}{2} (c_{\infty} + c_{0})$. We now consider $K_{X_0(p)} - \frac{p - 11}{12}(c_{\infty} + c_{0})$. This is a rational divisor, and has order dividing 6 in the Jacobian. By Ogg's conjecture, proven by Mazur, the rational torsion has order dividing the numerator of $\frac{p - 1}{12}$. By the congruence condition on $p$, the numerator is coprime to 6. This forces $K_{X_0(p)} - \frac{p - 11}{12}(c_{\infty} + c_{0})$ to have order 1, and so equality holds.

    The other cases follow similarly. When $p \equiv 7 \mod{12}$, $D_4 = 0$ and we note that $3K_{X_0(p)} - \left(\frac{p - 11}{4}(c_0 + c_{\infty}) - 2D_3\right)$ is 2-torsion and rational. By Mazur's rational torsion result, there is no non-trivial 2-torsion for $p \equiv 7 \mod{12}$. For $p \equiv 5 \mod{12}$, $D_3 = 0$ and again by Mazur's torsion result, we can divide by 3.
\end{proof}

A similar result holds for the canonical divisors of $X_0(N)$, with an identical proof.
\begin{proposition}
    For any integer $N$, the canonical divisor satisfies $6K_{X_0(N)} = D' - 4D_3 - 3D_4$, for some divisor $D'$ supported on the cusps.
\end{proposition}

\subsection{The Minimal Regular Models of $X_{0}(p)$ and $X_{0}^{+}(p)$}
A key ingredient of the proof of Theorem~\ref{THM1} is reduction modulo the only prime of bad reduction for $X_{0}(p)$ and $J_{0}(p)$, namely $p$. We give a brief description of the special fibres of the minimal regular models of $X_0(p)$ and $X_0^{+}(p)$ over $\mathbb{F}_{p}$.

The reductions of $X_{0}\left( p\right)$  and $X_{0}^{+} \left( p\right)$ are smooth away from $p$.
The special fibre of the minimal regular model of $X_{0}(N)$ at $p \vert N$, where $N$ is square-free and coprime to $6$, is described in \cite{deligne1973schemas} and \cite[Appendix]{mazur1977modular}. We recall this for prime level $N = p \ge 5$ 

\begin{theorem}
The special fibre of the minimal regular model of $X_{0} \left( p \right)_{/\mathbb{Z}_{p}}$ consists of two  copies of $\mathbb{P}^{1}_{/ \mathbb{\overline{F}}_{p}}$, each with multiplicity $1$, crossing transversely at the supersingular points.
More precisely, the crossing points are non-cuspidal points parametrising supersingular elliptic curves, with a specified subgroup of order $p$. There are precisely $k$ crossing points corresponding to supersingular elliptic curves with $j \ne 0, 1728$, where 
\begin{center}
    $k = \frac{p-a}{12}$ where $p \equiv a \mod{12}$.
\end{center}
If $j \in \{0, 1728\}$ is a supersingular $j$-invariant,  we replace the crossing by one projective line if $j = 1728$ and by a chain of two lines if $j = 0$, connecting the two components. \end{theorem}

Recall that $j=0$ is supersingular if and only if $p \equiv 1 \mod{6}$, and $j= 1728$ is supersingular if and only if $p \equiv 1 \mod{4}$.

The minimal regular model of $X_{0}^{+}\left( p \right)$ is obtained from that of $X_{0}\left( p\right)$ as described by Xue \cite{xue2009minimal}. 

 \begin{theorem} \label{minregmod}
  The special fibre of the minimal regular model of $X_{0}^{+}\left( p\right)_{/ \mathbb{Z}_{p}}$ over $\mathbb{F}_{p}$ is a rational curve with $\frac{1}{2}S_{p^{2}}$ nodes where $S_{p^{2}}$ is the number of supersingular $j$-invariants that are not in $\mathbb{F}_{p}$. 
 \end{theorem}

\section{Fixed Points of Hecke Operators}  \label{heckops}
In this section we describe how to compute the shadow points corresponding to Hecke operators of modular curves. These form the backbone of the proofs of Theorems \ref{THM1} and \ref{THM2}. More precisely, for $C = X_{0}(N)$, \textit{Atkin-Lehner operators} $w_{d}$ for all $d \vert N$ and co-prime to $\frac{N}{d}$, and for all $\ell \nmid N$ the \textit{Hecke correspondence} $T_{\ell}$ define standard correspondences over $C$.

As we now explain, the set of fixed points of Hecke correspondences contain CM points on  $X_{0}(N)$, that is, non-cuspidal points parametrising an elliptic curve with complex multiplication.

For $\phi = w_{d}$, $\phi = \phi^\vee$ and the degrees of the correspondences are 1. This gives the following expression for the shadow since the canonical divisor is invariant under automorphisms.
\begin{equation}
    \Sh(\phi) = (2g - 2)F_{\phi} - \deg(F) K_C.
\end{equation}
The fixed points of Atkin-Lehner operators are well-known due to Ogg \cite[Proposition 3]{ogg1974hyperelliptic}.
\begin{proposition}
The points in the support of $F_{w_{d}}$ are either:
\begin{itemize}
    \item cusps if $d = 4$;
    \item CM points with endomorphism ring containing $\mathbb{Z}[ \sqrt{-d}]$. \end{itemize}
    Moreover, all fixed points occur with multiplicity $1$.
\end{proposition}
In particular, for prime level $N= p$,  the fixed point divisor for $\phi = w_p$ is simply $D_{p} + D_{4p}$, where we recall our convention that if $p$ is not the discriminant of a CM order $D_p = 0$.

The correspondence $T_{\ell}$ where $\ell \nmid N$, can be described via a square as follows, where the vertical arrows are both the forgetful map
$$
\begin{tikzcd}
    X_0(N\ell) \arrow{r}{w_\ell}\arrow{d} & X_0(N\ell)\arrow{d} \\
    X_0(N) & X_0(N)
\end{tikzcd}
$$
see \cite[Chapter 5]{diamond2005dimension} for details.
As a result of this, we see that in the sense of the previous section, $T_{\ell} = T_{\ell}^{\vee}$. Further, as is well known, the degrees for both are $\ell + 1$. This simplifies $\Sh(T_\ell)$ to $(2g - 2)F_{T_{\ell}} - \deg(F_{T_{\ell}})K_{X_0(N)} - 2 T_{\ell}(K_{X_{0}(N)}) + 2(\ell + 1)K_{X_0(N)}$.

\begin{proposition}
    The points in the support of $F_{T_{\ell}}$ are either:
    \begin{itemize}
        \item Cusps;
        \item CM elliptic curves, with endomorphism ring $\mathcal{O}$ containing an element of norm $\ell$, and an ideal, $I$ such that $\mathcal{O} / I$ is cyclic of order $N$.
    \end{itemize}
\end{proposition}
\begin{proof}
    The cusps are fixed as there are 2 cusps of $X_0(N \ell)$ lying above each cusp of $X_0(N)$, and the Atkin-Lehner operator $w_\ell$ interchanges them. 

    Now suppose $(E, C_N)$ is a fixed point of the Hecke operator, then there exists an order $\ell$ subgroup, $C_\ell$, such that $(E, C_N) \cong (E / C_\ell, (C_N + C_\ell)/C_\ell)$. In particular, $E$ has an endomorphism of order $\ell$. The last condition is imposing the existence of a suitable order $N$ subgroup.
\end{proof}

It remains to compute the multiplicities with which these occur. There are two reasons for this multiplicity to be greater than one: the first is that the graph of $X_0(N\ell)$ in $X_0(N)^2$ is tangent to the diagonal here; the second is that the graph has a node at this point. We first exclude the first possibility.

\begin{proposition}
    Suppose $P \in X_0(N\ell)$ has both projections to $X_0(N)$ equal, then the image of the tangent space of $P$ under the map from $X_0(N\ell)$ to $X_0(N)^2$ is not diagonal.
\end{proposition}
\begin{proof}
    We start with the case that $P$ is a cusp. In this case, the local parameter at $P$ is either the same as for its image in $X_0(N)$, or its $\ell$-th root, and the opposite is true for $w_\ell(P)$. In either case, the projection of its tangent space to $X_0(N)^2$ is along one of the natural fibrations.

    Otherwise, $P$ corresponds to an elliptic curve. By the previous proposition, this elliptic curve has CM. Since the $j$-invariant is a local parameter for CM points on $X_0(N)$, it is enough to check that $j - j \circ \omega_\ell$ has simple zeros at CM points on $X_0(N\ell)$. As this function is a pullback from $X_0(\ell)$, we start by proving this for $X_0(\ell)$. For $\ell = 2$, $j - j \circ \omega_\ell$ has two poles of order 2, and so is a degree 4 function. There are 3 CM orders with an element of norm 2, $\mathbb{Z}[\sqrt{-1}]$, $\mathbb{Z}[\sqrt{-2}]$ and $\mathbb{Z}[\frac{-1 + \sqrt{-7}}{2}]$. The first two of these have a unique ideal of index 2, and the second has 2 choices, giving rise to a total of 4 points on $X_0(2)$ where this function vanishes. This forces all 4 points to be simple zeros. An identical counting argument applies for $\ell = 3$.

    For $\ell \geq 5$, we can reduce modulo $\ell$. On the regular model of $X_0(\ell)$ modulo $\ell$, the Atkin-Lehner involution acts by swapping the Frobenius and Verschiebung components, and maps a $j$-invariant on the Frobenius component to $\Fr(j)$, and similarly for the Verschiebung component. 

    In particular, the function $j - j \circ \omega_\ell$ vanishes on the points of each component with $\mathbb{F}_\ell$-rational $j$-invariant. As there are $l$ such $j$-invariants, there would be at least $2\ell$ zeroes if each appears on a separate component. However, there are rational supersingular $j$-invariants, which correspond to complex multiplication by $\mathbb{Z}[\sqrt{-p}]$ and $\mathbb{Z}\left[\frac{-1 + \sqrt{-p}}{2}\right]$. These CM curves reduce in pairs to the rational supersingular $j$-invariants, and so $j - j \circ \omega_\ell$ has all its zeroes in pairs, and one pair for each $j$-invariant over $\mathbb{F}_\ell$. In particular, it has at least $2\ell$ distinct zeroes, and has degree $2\ell$, and so all zeroes must be simple. 
\end{proof}

We now handle the other source of higher multiplicity, where multiple distinct points of $X_0(N\ell)$ map to the same point in $X_0(N)^2$.

\begin{proposition}
    The multiplicity of cusps in $F_{T_\ell}$ are 2, and for an elliptic curve with CM by $\mathcal{O}$, it is the number of prime ideals above $\ell$ in $\mathcal{O}$.
\end{proposition}
\begin{proof}
    The multiplicity is determined by number of points of $X_0(N\ell)$ above each point of the diagonal in $X_0(N)^2$.
    
    For each cusp of $X_0(N)$, there are exactly 2 cusps of $X_0(N\ell)$ above it, and both map to the same point in $X_0(N)^2$, namely the diagonal embedding of the cusp.

    For a point of $X_0(N\ell)$ corresponding to a CM elliptic curve, the points above it of interest, are the ones where the $\ell$-structure comes from the CM structure, and so the $\ell$-structure comes from a choice of ideal of index $\ell$. Any such choice is appropriate, as if one ideal is principal, so is the other.
\end{proof}

Combining the above in the special case of $T_2$, gives the following 
\begin{center}
$F_{T_2} = D_4 + D_8 + 2D_7 + 2D_{\infty}$
\end{center}
where $D_{\infty}$ is the cuspidal divisor (all with multiplicity $1$) and we highlight the fact that if no such CM point exists, $D_d$ is assumed to be zero.

\section{Proof of Theorem 1}
\label{proofof1}
Throughout this section we write $p \ge 5$ for a prime number, $X_{0}(p)$ for  the complete modular curve associated to the congruence subgroup $\Gamma_{0}(p)$ and $X = X_{0}^{+}(p)$ for the quotient curve obtained from the unique Atkin-Lehner involution $w_{p}$. We also write $w_{p}$ for the induced involution on the Jacobian $J_{0}(p)$ of $X_{0}(p)$ and $J = J_{0}^{+}(p)$ for the quotient variety. In this case we are able to identify $J$ with the Jacobian of $X$ since the quotient map  $ \pi : X_{0}(p) \rightarrow X $ is ramified, see the proof of \cite[Lemma 13.8]{mazur1977modular}. 

The proof of Theorem \ref{THM1} reduces to proving that the scaled shadow of some carefully chosen Hecke correspondence defines a point of infinite order in the Jacobian. We consider various points, principally $\Sh(T_2)$, but also $\Sh(T_3)$ and $\Sh(w_p)$ depending on the congruence class of $p$. We denote this point by $D$ and note the following trivial corollary of Theorems \ref{M1} and \ref{M2}. 
\begin{corollary}
 Suppose $[\pi( D )] \in J(
 \Q)$ is non-zero. Then $[D] \in J_{0}(p)( \Q )$ is of infinite order. 
\end{corollary}

\subsection{Reduction Arguments} For most congruence classes of $p$ we map the shadow point to $J$ and reduce modulo $p$ to conclude that it is non-zero. In particular, we use the following results. 
\begin{proposition}
There exists a homomorphism
\begin{center}
    $\mathrm{red} : J(\Q_{p}) \longrightarrow \Pic^{0}(\mathcal{X}_{/\mathbb{F}_{p}})$
\end{center}
where $ \mathcal{X}_{/\mathbb{F}_{p}}$ is the special fibre of the minimal regular model of $X_{/ \mathbb{\Q}_{p}}$. 
\label{prop:redmap}
\end{proposition}
\begin{proof}
 Let $\mathcal{J}/\mathbb{Z}_{p}$ be the N\'eron model of $J_{/ \mathbb{Q}_{p}}$ and write  $\mathcal{J}_{/ \mathbb{F}_{p}}$ for its special fibre. By the N\'eron mapping property 
\begin{center}
    $J(\Q_{p}) \cong \mathcal{J}( \mathbb{Z}_{p})$.
\end{center}
We write $\mathcal{J}^{0}$ for the connected component of the identity and $\Phi$ for the component group. By definition
\begin{center}
 $\mathcal{J}/\mathcal{J}^{0}\cong \Phi$,
\end{center}
and since $J$ is the Jacobian of a curve, we can determine the component group from the special fibre of the minimal regular model of the curve, see \cite[Section 9.6, Theorem 1]{bosch2012neron}. As the special fibre described in Theorem \ref{minregmod} has a single component we deduce that $\Phi$ is trivial, and hence $\mathcal{J}(\mathbb{F}_{p}) = \mathcal{J}^{0}(\mathbb{F}_{p}) $. Moreover, by \cite[Section 9.5, Theorem 1]{bosch2012neron}, $\mathcal{J}^{0}(\mathbb{F}_{p}) \cong \Pic^{0}(\mathcal{X}_{/ \mathbb{F}_{p}})$ and $\mathrm{red}$ is simply the composition 
\begin{center}
    $J(\Q_{p}) \xrightarrow{\sim} \mathcal{J}(\mathbb{Z}_{p}) \rightarrow \mathcal{J}(\mathbb{F}_{p}) = \mathcal{J}^{0}(\mathbb{F}_{p})\xrightarrow{\sim} \Pic^{0}(\mathcal{X}_{/ \mathbb{F}_{p}})$.
\end{center}
\end{proof}
We now give a criterion for a divisor to be non-vanishing in $\Pic^0(\mathcal{X} / \mathbb{F}_p)$.
\begin{proposition}
    Let $D = \sum_i n_i a_i$ be a divisor on $\mathcal{X} / \mathbb{F}_p$, and define $f(x) = \prod_i (x - a_i)^{n_i}$, by identifying the special fibre of $X$ with a $\mathbb{P}^1$. If $j \in \mathbb{F}_{p^2} \setminus \mathbb{F}_p$ is a supersingular $j$-invariant such that $f(j) \notin \mathbb{F}_p$, then $D \neq 0$. 
    \label{prop:triv}
\end{proposition}
\begin{proof}
    As $j$ is a supersingular $j$-invariant, in $\mathcal{X} / \mathbb{F}_p$ it is glued to $\bar{j}$, where $\overline{\phantom{j}}$ denotes the action of Frobenius. Since $f(j) \notin \mathbb{F}_p$, then $f(j) \neq \overline{f(j)} = f\left(\overline{j}\right)$ . In particular, $f$ is not a well-defined function on $\mathcal{X} / \mathbb{F}_p$, and so $D$ is not trivialised by it. 
\end{proof}

The next proposition describes one way of applying the previous proposition in characteristic 0. 

\begin{proposition}
    Let $D = \sum_i n_i P_i$ be a rational, degree 0 divisor on $X_0(p)$. Take a CM order, $\mathcal{O}$ such that $p$ is inert in both $\mathcal{O}$ and the Hilbert ring class field of $\mathcal{O}$. Define $\alpha = \prod_i (j_{\mathcal{O}} - j(P_i))^{n_i}$, where $j_{\mathcal{O}}$ is a $j$-invariant attached to $\mathcal{O}$. If there does not exist an $a \in \mathbb{Z}$ such that the $\mathfrak{p}$-adic valuation of $\alpha - a$ is strictly positive for some $\mathfrak{p}$ lying over $p$, $\pi(D)$ is non-zero in $J$.
    \label{prop:CMred}
\end{proposition}
\begin{proof}
    We first push $D$ forwards under $\pi$, to obtain an element of $J$. By Proposition~\ref{prop:redmap}, we have a reduction map $J \to \Pic^{0}(\mathcal{X} / \mathbb{F}_p)$. Under this map, $D$ is sent to $\sum_i n_i j(P_i)$. Due to the assumptions on $\mathcal{O}$, $j_{\mathcal{O}}$ is a supersingular $j$-invariant mod $\mathfrak{p}$, which is not in $\mathbb{F}_p$. As $\alpha = f(j_{\mathcal{O}})$, where $f$ is as in the previous proposition, the assumption on the $\mathfrak{p}$-adic valuation translates into $f(j_{\mathcal{O}}) \notin \mathbb{F}_p$, and so the reduction of $\pi(D)$ is non-zero. 
\end{proof}

\subsection{Primes congruent to 5 or 7 modulo 12} \label{57mod12sec}
We work with the divisor $D = 3 \cdot \Sh(T_{2})$ or $6 \cdot \Sh(T_{2})$ depending on the canonical divisor and apply Proposition~\ref{prop:CMred} to $D$. The explicit dependence of the coefficients of $D$ on $p$ poses an obstacle. Otherwise, the quantity $\alpha$ could be computed, without use of $p$, for some choice of $\mathcal{O}$ which works for the whole congruence class. We circumvent this issue by observing that as $X$ has non-split semistable reduction, the group $\Pic^0(\mathcal{X} / \mathbb{F}_p)$ has exponent $p + 1$, and so we may reduce the coefficients of $D$ modulo $p + 1$ for the purposes of this calculation. 

For each congruence class of $p$ within these classes, we can now compute the value $\alpha$, and determine which primes do not have the necessary valuation condition. This can be achieved by factorising $\alpha - \sigma(\alpha)$, where $\sigma$ is a Frobenius element of the Galois group for $p$. In practice, this factorisation can be unwieldy. We avoid this by using two CM orders, and taking the greatest common divisor of the two values, and then factorising. In Table~\ref{table:2}, we list the divisors $D$, and  in Table~\ref{table:1} we list the CM orders with the corresponding $j$-invariants used, along with any primes where Proposition~\ref{prop:CMred} does not apply. The exceptions can all be seen to be primes where $X_0(p)$ is hyperelliptic. These were computed using the \texttt{jTest} routine in \texttt{divisor\_checker.m}.

\begin{table}[ht]
\begin{tabularx}{\textwidth}{|X|X|}
\hline 
    $p$ & $D$ \\
    \hline 
     $p \equiv 7 \mod{12}$, $p \equiv 7 \mod{8}$, $ p \equiv 3,5,6 \mod{7}$ & $3\Sh(T_{2}) =  12D_{12} - 4D_{3} - 8(c_{0} + c_{\infty})$  \\
     \hline 
     $p \equiv 7 \mod{12}$, $p \equiv 7 \mod{8}$, $ p \equiv 1,2,4 \mod{7}$ & $3\Sh(T_{2}) = \left( p  - 19 \right) D_{7} + 12D_{12} + 4 D_{3} +  (-p + 3) \left( c_{0} + c_{\infty} \right)$ \\
     \hline 
     $p \equiv 7 \mod{12}$, $p \equiv 3 \mod{8}$, $p \equiv 3,5,6 \mod{7}$ & $6 \Sh(T_{2}) = \left( p - 19 \right)D_{8} + 24 D_{12} + (-p -5)( c_{0} + c_{\infty})$ \\ \hline
 $p \equiv 7 \mod{12}$, $p \equiv 3 \mod{8}$, $ p \equiv 1,2,4 \mod{7}$ &  $6 \Sh(T_{2}) = (p-19)D_{8} + (2p -38)D_{7}  + 16 D_{3}+ 24D_{12} + (-3p + 17)(c_{0} + c_{\infty})$ \\
 \hline 
  $p \equiv 5 \mod{12}$, $ p \equiv 5 \mod{8}$, $ p \equiv 3,5,6 \mod{7}$
& $ 6 \Sh(T_{2}) = (p-11)D_{4} + 12D_{16}+ ( -p -1) (c_{0} + c_{\infty})$ \\
  \hline
  $p \equiv 5 \mod{12}$, $ p \equiv 5 \mod{8}$, $ p \equiv 1,2,4 \mod{7}$& $ 6 \Sh(T_{2}) = (p+1)D_{4} + 2(p-17) D_{7}+ 12D_{16} + (-3p + 21)(c_{0} + c_{\infty})$  \\
   \hline 
   $p \equiv 5 \mod{12}$, $p \equiv 1 \mod{8}$,  $ p \equiv 3,5,6 \mod{7}$ & $6 \Sh(T_{2}) = (p-5)D_{4} + (p-17)D_{8}+ 12D_{16} + (-2p + 10)(c_{0} + c_{\infty})$ \\ \hline 
    $p \equiv 5 \mod{12}$, $p \equiv 1 \mod{8}$,  $ p \equiv 1,2,4 \mod{7}$  & $ 6 \Sh(T_{2}) = (p + 7) D_{4} + 2(p-17) D_{7} + (p-17)D_{8} + 12D_{16} + (-4p + 32)(c_{0} + c_{\infty})$ \\ \hline 
    $p \equiv 1 \mod{12}$, $p \equiv 5 \mod{8}$, $p \equiv 3,5,6 \mod{7}$ & $6 \Sh(T_{2}) = (p-19)D_{4} + 12D_{16} + 24D_{12} + (-p-17)(c_{0} + c_{\infty})$ \\ 
\hline 
$p \equiv 1 \mod{12}$, $p \equiv 5 \mod{8}$, $p \equiv 1,2,4 \mod{7}$ & $ 6 \Sh(T_{2}) = (p-7)D_{4} + 2(p-25) D_{7} + 16 D_{3} + 12D_{16} + 24D_{12} + (-3p + 5) (c_{0} + c_{\infty} )$ \\
\hline 
$p \equiv 1 \mod{12}$, $p \equiv 1 \mod{8}$, $p \equiv 3,5,6 \mod{7}$ & $6 \Sh(T_{2}) = (p-13)D_{4} + (p-25)D_{8} + 8D_{3} + 24D_{12} + 12D_{16} + (-2p -6)(c_{0} + c_{\infty})$ \\ \hline
$p \equiv 1 \mod{12}$, $p \equiv 1 \mod{8}$, $p \equiv 1, 2, 4 \mod{7}$ & $6\Sh(T_2) = 24D_3 + (p - 1)D_4 + 2(p - 25)D_7 + (p - 25)D_8 + 24D_{12} + 12 D_{16} + (-4p + 16)(c_0 + c_\infty)$  \\ \hline
\end{tabularx}
\caption{Scaled Shadow Points}
\label{table:2}
\end{table}

\begin{table}[ht]
\begin{tabularx}{\textwidth}{|X|X|c|c|}
\hline
$p$ & $j$ & $\Delta_j$ & Exceptions \\ \hline
$p \equiv 7 \mod{12}$, $p \equiv 7 \mod{8}$, $p \equiv 3, 5, 6 \mod{7}$ & $ 76771008 + 44330496\sqrt{3} $ & $-36$ & $7, 31$ \\ \hline
$p \equiv 7 \mod{12}$, $p \equiv 7 \mod{8}$, $p \equiv 1,2,4 \mod{7}$ & $ 76771008 + 44330496\sqrt{3} $ & $-36$ & $\emptyset$ \\ \hline
$p \equiv 7 \mod{12}$, $p \equiv 3 \mod{8}$, $p \equiv 3,5,6 \mod{7}$ & $ 76771008 + 44330496\sqrt{3} $ & $-36$ & $19$ \\ \hline
$p \equiv 7 \mod{12}$, $p \equiv 3 \mod{8}$, $p \equiv 1,2,4 \mod{7}$ & $76771008 + 44330496\sqrt{3}$, $41113158120 + 29071392966\sqrt{2}$ & $-36$, $-64$ & $\emptyset$ \\ \hline
$p \equiv 5 \mod{12}$, $p \equiv 5 \mod{8}$, $p \equiv 3,5,6 \mod{7}$ & $1417905000 + 818626500\sqrt{3}$ & $-48$ & 5 \\ \hline
$p \equiv 5 \mod{12}$, $p \equiv 5 \mod{8}$, $p \equiv 1,2,4 \mod{7}$ & $1417905000 + 818626500\sqrt{3}$, $26125000 + 18473000\sqrt{2}$ & $-48$, $-32$ & $29$\\ \hline
$p \equiv 5 \mod{12}$, $p \equiv 1 \mod{8}$, $p \equiv 3,5,6 \mod{7}$ & $1417905000 + 818626500\sqrt{3}$, $137458661985000 - 51954490735875\sqrt{7}$ & $-48$, $-112$ & $17$, $41$ \\ \hline
$p \equiv 5 \mod{12}$, $p \equiv 1 \mod{8}$,  $ p \equiv 1,2,4 \mod{7}$ & $1417905000 + 818626500\sqrt{3}$, $-17424252776448000 + 3802283679744000\sqrt{21}$ & $-48$, $-147$ & $\emptyset$ \\ \hline
$p \equiv 1 \mod{12}$, $p \equiv 5 \mod{8}$, $p \equiv 3,5,6 \mod{7}$ & $26125000 + 18473000\sqrt{2}$ & $-32$ & $13$ \\ \hline
$p \equiv 1 \mod{12}$, $p \equiv 5 \mod{8}$, $p \equiv 1,2,4 \mod{7}$ & $26125000 + 18473000\sqrt{2}$, $188837384000 + 77092288000\sqrt{6}$ & $-32$, $-72$ & $37$ \\ \hline
$p \equiv 1 \mod{12}$, $p \equiv 1 \mod{8}$, $p \equiv 3,5,6 \mod{7}$ & $\frac{1}{2} \times 3^3 \times 5^3 \times ( 11196968150480235272373040  - 4232056166297363790876672 \sqrt{7} -7917452107934369456841021\sqrt{2} + 2992515613551209199902175 \sqrt{14})$ & $-448$ & $\emptyset$ \\\hline
\end{tabularx}
\caption{Supersingular $j$-invariants.}
\label{table:1}
\end{table}

\subsection{Primes congruent to 1 modulo 12} \label{1mod12sec}
The methods of the previous section continue to apply for $p \equiv 1 \mod{12}$, except for the class $p \equiv 1 \mod{8}$ and $p \equiv 1, 2, 4 \mod{7}$. For primes in this class, there is no longer a uniform choice of CM order with the necessary properties. The shadows are listed in Table~\ref{table:2}, and for the other 3 congruence classes, the choice of $\mathcal{O}$ and $j$ are in Table~\ref{table:1}. 

To handle the remaining cases, we show that at least one of $\Sh(T_2)$ and $\Sh(T_3)$ are of infinite order.

Let $f_1$ and $f_2$ be the functions attached to the pushforwards of $6\Sh(T_2)$ and $6\Sh(T_3)$ respectively, as in Proposition~\ref{prop:triv}. As before, we reduce the coefficients of $6 \Sh(T_2)$ and of $6 \Sh(T_3)$ modulo $p + 1$, as this does not affect their reduction. Define the curve $\tilde{C}_i$ to be the fibre product $\mathbb{P}^1 \times_{f_i} \mathbb{P}^1$ less the diagonal. Further, let $C_i$ be the quotient of $\tilde{C}_i$ under the involution swapping the factors and assume $C_i$ is embedded in $\mathbb{P}^2 = \mathrm{Sym}^2 \mathbb{P}^1$.
\begin{proposition}\label{prop:intsn_gen}
    Assume that the pushforwards of $6\Sh(T_2)$ and $6\Sh(T_3)$ are both zero under the reduction map in Proposition~\ref{prop:redmap}. Then $\# C_1(\mathbb{F}_p) \cap C_2(\mathbb{F}_p) \geq g(X_0^+(p))$, with notation as above.
\end{proposition}
\begin{proof}
    As the pushforward of $6\Sh(T_2)$ is zero under the reduction map, for every supersingular $j$-invariant, $f_1(j) \in \mathbb{F}_p$ by Proposition~\ref{prop:triv}. In particular, $f_1(j) = f_1(\overline{j})$, and so $(j, \overline{j})$ is a point on the fibre product. If $j$ is not defined over $\mathbb{F}_p$, this point is not on the diagonal, and so defines a point on $\tilde{C}_1$. This reduces to a point on $C_1$, and as $j$ is defined over $\mathbb{F}_{p^2}$, it becomes rational. The same argument applies for $f_2$ and so this point appears on both curves. The bound in the statement of the proposition then follows as there are $g(X_0^+(p))$ such pairs of $j$-invariants.
\end{proof}

Unless $C_1$ and $C_2$ share a common component, the size of the intersection is bounded by Bezout's theorem. We therefore study the components of $C_1$.
\begin{proposition}
    All components of $C_1$ contain the image of the point $(1728, 8000) \in \tilde{C}_1$.
\end{proposition}
\begin{proof}
    From Table~\ref{table:2}, we see that
    $$f_1 = \frac{x^{24} (x - 54000)^{12} (x - 287496)^{12}}{(x - 1728)^2 (x - 8000)^{26} (x + 3375)^{52}}.$$
    This is a perfect square, so let $f_1 = g_1^2$. This forces $\tilde{C}_1$ to have at least 2 components, one where $g_1$ agrees on the two points, and the other where $g_1$ takes opposite signs. These are stable under swapping the order of points and so $C_1$ has at least 2 components. 

    In fact, $\tilde{C_1}$ and all its reductions modulo suitable primes $p$ has exactly 2 geometric components. Since the given point lies on both of these ($1728$ and $8000$ are poles for $g_1$), the result follows.

    Let $G$ be the geometric Galois group of $f_1 : \mathbb{P}^1 \to \mathbb{P}^1$. This naturally acts on a set of size $\deg f_1 = 80$. The number of geometrically irreducible components of $\tilde{C_1}$ is the number of orbits of $G$ on ordered distinct pairs of elements. If $G$ is as large as possible (given the factorisation), i.e. $G \cong S_{40} \wr S_2$, then this action has exactly two orbits: pairs belonging to the same partition and pairs belonging to opposite partitions.

    The first step in computing the geometric Galois group of $f_1$ is to compute the geometric Galois group $H$ of $g_1$. This action is primitive by considering the monodromy element around $\infty$: its cycle decomposition is $(1, 13, 26)$. Any partition preserved by $H$ is preserved by this element. In particular, the block containing the fixed point must be fixed. This block must also contain either the $13$-cycle or the $26$-cycle, giving block sizes of $14$ or $27$, neither of which is a divisor $40$. This leaves only the trivial partitions. The function $g_1$ also has 5 simple branch points, and so $H$ is a primitive group containing a transposition. By Jordan's Theorem \cite[Theorem 3.3A]{DixMor}, $H = S_{40}$.

    Now for $f_1$, $G$ is contained in $S_{40} \wr S_2$ (the largest group preserving the block structure). The monodromy of $f_1$ above $\infty$ has cycle type $(2, 26, 52)$, which squares to two copies of the cycles for $g_1$. The pairs of cycles can be split up by which blocks of the partition they act on. The image of the 5 simple branch points of $g_1$ under squaring are also 5 simple branch points for $f_1$. This shows that $G$ contains a transposition, which necessarily permutes just one block. Combining this transposition with the monodromy at infinity shows $S_{40} \times S_{40}$ is a subgroup of $G$. Finally, as the monodromy element above $\infty$ swaps the blocks, $G$ is larger than $S_{40} \times S_{40}$, but $S_{40}^2$ is an index 2 subgroup of $S_{40} \wr S_2$. Hence $G$ is $S_{40} \wr S_2$.

    This description only depended on the description of the monodromy element at $\infty$ and the existence of a simple branch point. These can only change if $p$ divides the difference of two of the ramification points above $\infty$, or if the simple branch points collide with others. The first cannot occur for any $p$ of interest (e.g. if $p$ is 1 mod 24). By similarly considering divisibility of the differences between the ramification points above $0$, this also cannot change by reduction. This leaves only the 5 simple branch points colliding with each other. For there not to be a transposition, at least 3 of the ramification points must end up in the same fibre. Such a prime must divide the discriminant of the polynomial defining the branch points and also the discriminant of the derivative of this polynomial. Again, no prime satisfying the congruence conditions exists.
    
    We therefore conclude that all irreducible components of $C_1$ over $\overline{\mathbb{Q}}$ remain irreducible modulo $p$, and so all irreducible components modulo $p$ contain the given point.
\end{proof}

The value of $6 \Sh(T_3)$ depends on the congruence class of $p \mod{11}$. If $p$ is a quadratic non-residue modulo 11,  
$$6 \Sh(T_3) = (p - 1)D_3 + 12 D_4 + 2(p - 25)D_8 + (p - 25)D_{12} + 24 D_{27} + 12 D_{36} - 4(p - 4)(c_{\infty} + c_{0}),$$
and 
$$f_2 = \frac{(x - 1728)^{12} (x + 12288000)^{24} (x^2 - 153542016x - 1790957481984)^{12}}{x^2 (x - 8000)^{52} (x - 54000)^{26}}.$$
If $p$ is a residue, 
$$6 \Sh(T_3) = (p + 15)D_3 + 24 D_4 + 2(p - 25)D_8 + 2(p - 25)D_{11} + (p - 25)D_{12} + 24 D_{27} + 12 D_{36} - (6p - 38)(c_{\infty} + c_{0}),$$
and
$$f_2 = \frac{x^{12} (x - 1728)^{24} (x + 12288000)^{24} (x^2 - 153542016x - 1790957481984)^{12}}{(x - 8000)^{52} (x + 32768)^{52} (x - 54000)^{26}}.$$

In either case, we see that the image of $(1728, 8000)$ cannot be on $C_2$, since 1728 is a zero of $f_2$ in both cases, and 8000 is a pole. In particular, $C_2$ cannot contain a component of $C_1$, since all components of $C_1$ contain the image of $(1728, 8000)$. We can now apply Bezout's theorem. If $p$ is not a residue modulo 11, then the degrees of $C_1$ and $C_2$ are both 79 (the degrees of $f_1$ and $f_2$, minus 1 for the diagonal component that has been removed), and so $\# C_1 \cap C_2 \leq 79^2$. Similarly, if $p$ is a residue, then $\# C_1 \cap C_2 \leq 79 \times 129$. Combining these inequalities with the inequality of Proposition~\ref{prop:intsn_gen} gives an upper bound for the genus of any curve with both $\Sh(T_2)$ and $\Sh(T_3)$ torsion. 

As $g(X_0^+(p)) = \frac{p - 1}{24} - \frac{1}{4}h(-4p)$, any $p$ with both divisors torsion has $h(-4p) \geq \frac{p - 1}{6} - 4 \times 79^2$ if $p$ is a non-residue modulo 11, and $h(-4p) \geq \frac{p - 1}{6} - 4 \times 79 \times 129$ otherwise. Applying the class number bound, Proposition~\ref{classbound}, gives $2 \sqrt{p} \log(4p) \geq \frac{p - 1}{2} - 12 \times 79^2$, which gives $p \leq 172090$. There are 430 primes in this congruence class where the genus is below the bound, and the largest is 150889. For each prime, we check that $6 \Sh(T_2) \neq 0$ using Proposition~\ref{prop:triv}.

Similarly, if $p$ is a residue, the bound is $p \leq 273684$, and there are 642 primes in this class with the genus below the bound, and of these, 244897 is the largest. Again, applying Proposition~\ref{prop:triv} shows that for all of them, $6 \Sh(T_2)$ is non-zero. These cases are checked in the file \texttt{primes1T2.m}.

\subsection{Primes congruent to 11 modulo 12} \label{11mod12sec}
When $p \equiv 11 \mod{12}$ we cannot guarantee that $\Sh(T_{2}) \ne 0$, or indeed $\Sh(T_\ell)$ for any fixed prime $\ell$. We instead consider $\Sh(w_{p})$, 
\begin{center}
    $\Sh(w_{p}) = \left( g -1 \right) \left( D_{p} + D_{4p} \right) - n(g-1)\left( c_{0} + c_{\infty} \right)$
\end{center}
where $n = \frac{1}{2} \left( h \left( -p \right) + h \left( - 4p \right) \right)$, and $g$ is the genus of $X_{0}(p)$.

If $[\Sh(w_{p})]$ is torsion in $J_{0}(p)(\mathbb{Q})$, we divide by $g-1$ and project onto $X_{0}^{+}( p)$ to obtain 
\begin{center}
$ \left( \pi(D_{p}) + \pi(D_{4p}) \right) - 2n c_{\infty}$
\end{center}
where $c_{\infty}$ is the unique cusp of $X_{0}^{+}(p)$.
Applying Proposition \ref{M3} we get:
\begin{center}
    $h \left( -4p \right)  + h \left( -p \right) \ge \frac{p}{96}$.
\end{center}

When $p \equiv 7 \mod{8}$, $h \left( -4p \right) = h \left( -p \right)$ and we find that for $p \ge 749, 504$:
\begin{center}
    $2 h \left( -p \right) \le \frac{2}{3} \sqrt{p} \log p < \frac{p}{96}$.
\end{center}
Computationally, we reduce the above bound to $104 711$ and note that there are $566$ primes satisfying the stated congruence conditions. 

When $p \equiv 3 \mod{8}$, $h \left( - 4p \right) = 3 h \left( -p \right)$, and repeating the above we get $p \ge 3 754 815$, which we reduce to $49 451$, and $266$ primes to check.
The list of exceptions includes the primes $\{ 11, 23, 47, 59, 71 \}$, for which $X_{0}(p)$ is hyperelliptic. For all non-hyperelliptic exceptions we find a supersingular $j$-invariant defined over $\mathbb{F}_{p^{2}} \setminus \mathbb{F}_{p}$ and argue as in Proposition~\ref{prop:triv}. These cases are checked in the file \texttt{primes11AL.m}, using that the $j$-invariants of points in $D_p$ and $D_{4p}$ reduce to the $\mathbb{F}_p$-rational supersingular $j$-invariants modulo $p$. A suitable $j$-invariant cannot be found for the bi-elliptic $X_0(p)$, and so are checked individually using the $T_2$ shadow.

\section{Proof of Theorem 2}
\label{proofof2}
In this section we prove that the Ceresa cycle is non-trivial for all but finitely many $X_{0}(N)$, where $N $ is any positive integer. Our proof uses the following elementary result. 

\begin{proposition} \label{covers}
Let $C$ and $D$ be algebraic curves of genus at least $3$, such that $C$ covers $D$ and $\Cer(D, f)$ is non-zero in $\CH_{1}(J_D)$ for all choices of $f \in \CH_1(D)$. Then $\Cer(C, e)$ is also non-zero for any choice of basepoint divisor $e$.
\label{prop:pushforward}
\end{proposition}
\begin{proof}
Let $\pi : C \to D$ denote the covering map. As $\pi_* [\iota_{e}(C)] = \deg(\pi) [\iota_{\pi(e)}(D)]$, we have $\pi_* \Cer(C, e) = \deg(\pi) \Cer(D, \pi(e))$. By assumption, $\Cer(D, f)$ is non-torsion for all base-points, $f$, and so the pushforward of $\Cer(C,e)$ is of infinite order, and in particular, $\Cer(C,e)$ is also of infinite order. 
\end{proof}

For any prime $p$ dividing $N$, the natural quotient map induces a covering of curves
\begin{center}
    $X_{0}(N) \longrightarrow X_{0} \left( p\right)$
\end{center}
and thus from Theorem \ref{THM1}, it follows that if $N$ is divisible by a prime $p$ and 
\begin{center}
$p \notin \{2,3,5,7,11,13,17,19,23,29,31,37,41,47,59, 71 \} $
\end{center}
 that is, $X_{0}(p)$ has genus  at least $3$ and it is not hyperelliptic (see \cite{ogg1974hyperelliptic}), then $\Cer(N)$ is of infinite order. It remains to show that for any $p$ in the above list, $\Cer(p^n)$ is non-torsion for some $n \ge 2$. For most primes we take $n=2$. When $p = 2, 3,5,7$, we take larger powers to ensure that the Jacobian is not hyperelliptic and usually also of positive rank. 

The case for $p=2$ is well known. 
\begin{theorem} 
The cycle $\Cer(64)$ is non-zero.
\end{theorem}
\begin{proof}
 It is a well-known fact that the modular curve $X_{0}(64)$ is the Fermat quartic. Non-vanishing was proved by Bloch \cite{bloch1984algebraic}. 
\end{proof}
\subsection{Heegner Divisors}
For some $N \equiv 11 \mod{12}$, certain shadows are supported on Heegner points and cusps only, in which case we deduce that the (scaled) shadow point is non-torsion using the classical theory of Heegner divisors.

Suppose that $N$ is a positive integer and $K$ is an imaginary quadratic field whose discriminant is co-prime to $N$ and such that all primes dividing $N$ split in $K$. Then $P \in X_{0}(N)$ is a \textit{Heegner point} if it is a CM point with CM by an order in $K$. The \textit{Heegner divisor} defined by $P$ is the degree $0$ divisor 
\begin{center}
    $D_{P} = \displaystyle \sum_{\sigma} P^{\sigma} - M c_{\infty}$
\end{center}
where the sum is taken over the $\mathbb{Q}$-Galois conjugates of $P$, $c_{\infty}$ is the $\infty$ cusp of $X_{0}(N)$ and $M$ is the degree of the field over which $P$ is defined. The following is a simple consequence of results proved in \cite{gross1986heegner}.
\begin{theorem}[Gross--Zagier]
Let $P$ be as above, and let $D'_P = \sum_{\sigma} P^{\sigma} - M' c_{\infty}$, where $\sigma$ runs over the $K$-Galois conjugates of $P$ and $M'$ is the degree of the field extension $K(P)/K$. Assume there exists a cusp form of level $N$, weight 2 and rank 1. Then the $f$-isotypic component $[D'_{P, f}]$ is a point of infinite order in $J_{0}(N)(\Q)$. 
\end{theorem}
 \begin{corollary}
     Under the same assumptions as above, the divisor $D_{P}$ is of infinite order.
 \end{corollary}
 \begin{proof}
     Applying the Fricke involution  $w_N$ to $D'_P$ gives $\sum_{\sigma} P^{\sigma} - M' c_{0}$, and hence the sum of $D'_P$ and its Fricke conjugate recovers $D_P$ up to a cuspidal divisor. As the cuspidal divisor is torsion, it suffices to consider $(1 + w_N)D'_P$. Taking the $f$-isotypic component gives $2D'_{P, f}$ as $f$ has Fricke sign $+1$. By the above, this is of infinite order.
 \end{proof}

\begin{proposition}
For $p \in \{23, 47, 59, 71 \}$, $\Cer(p^2)$ is non-zero. 
\end{proposition}
\begin{proof}
For $p \in \{23, 47, 59, 71 \}$,  $X_{0}(p^2)$ has $p+1$ cusps, $c_{\infty}, c_{0}, c_{1}, \ldots, c_{p-1}$, where $c_{\infty}$ and $c_{0}$ are defined over $\mathbb{Q}$ and $c_{1}, \ldots, c_{p-1}$ form a single Galois orbit. We write $D$ for the divisor class corresponding to the sum of the cusps. We compute the following shadows:
\begin{itemize}
  \item $ N = 23^2 : 6\Sh(T_{2}) = 816D_{7} - 68D$;
\item $N = 47^2: 6\Sh(T_{3})= 3936D_{11} - 164D$; 
  \item $  N = 59^2: 6\Sh(T_{5}) =6360D_{11} - 212D $;
  \item $ N = 71^2: 6\Sh(T_{2}) = 9360D_{7} - 260D$.
\end{itemize}
The scaled shadow points differ from a Heegner divisor by a divisor supported on cusps only, and as such divisors are torsion points in $J_{0}(N)$ by the theorems of Manin \cite{Manin} and Drinfeld \cite{Drinfeld}, we conclude that the scaled shadow points are necessarily non-torsion.
\end{proof}

\subsection{Explicit Calculations} For curves of low genus, we can use \texttt{Magma} to efficiently compute divisors on the curve and explicitly test whether our chosen shadow point is torsion. 

\begin{proposition}
For $N \in \{11^2$, $5^3$, $13^2 \}$, the shadow $D = \Sh(w_{N})$ is non-torsion. 
\end{proposition}
\begin{proof}
   The \texttt{Magma} code used for these computations can be found in the online repository. We give a brief overview of our calculations below. 

We begin by computing a canonical model for the curve using a basis of the space of weight $2$ cusps forms for $\Gamma_{0}(N)$, see  \cite{galbraith1996equations}. In particular, we choose a basis whose elements are invariant or anti-invariant under the action of $w_{N}$ and from this we are able to easily construct $D$ explicitly. We show that modulo different primes, the divisor has different orders, and hence by the injectivity of torsion \cite[Appendix]{katz}, we conclude that the shadow is necessarily of infinite order.
\end{proof}

\begin{remark}
 We also verified that $X_{0}(74)$ has non-trivial Ceresa cycle by the same methods. This, along with Theorem \ref{THM1} and the proof of Theorem \ref{THM2} recovers the main result of \cite{kerr2024non}.
\end{remark}
 
 \subsection{Period Lattice Computations}
For $N \in \{ 3^5, 7^3, 17^2, 19^2, 29^2, 31^2, 37^2,41^2\}$, the genus of $X_{0}(N)$ is large and computing explicit divisors in \texttt{Magma} is extremely inefficient. Instead, we use the analytic structure of the curve and its Jacobian, and project our shadow point into a (complex) subvariety of dimension $1$, $2$ or $3$ corresponding to a subset of newforms of level $\Gamma_0(N)$ defining abelian sub-varieties of $J_0(N)$. We give a short overview of the analytic definition of the Jacobian $J_{0}(N)$ and its modular sub-varieties, see \cite[Chapter 6]{diamond2005dimension}  or \cite[Chapter 2]{cremona1997algorithms} for comprehensive descriptions.

Let $f_{1}, \ldots, f_{g}$ be a basis for $S_{2}(N)$, the space of weight 2 cusp forms for $\Gamma_{0}(N)$, and take $\gamma_{1}, \ldots, \gamma_{2g}$ to be a basis for the integral singular homology $H_{1}(X_{0}(N), \mathbb{Z})$. The \textit{period matrix} defined with respect to these bases is the $2g \times g$ complex matrix 
\begin{center}
    $\Omega = ( \omega_{jk}) = ( \langle \gamma_{j}, f_{k} \rangle)$
\end{center}
where $\langle \gamma_{j}, f_{k} \rangle = \int_{\gamma_{j}} 2 \pi i f_{k} (z) dz $. This integral is well-defined by the general theory of periods of cusp forms. The $2g$ rows of $\Omega$ are linearly independent over $\mathbb{R}$  and so their $\mathbb{Z}$-span defines a  lattice of rank $2g$ in $\mathbb{C}^{g}$, which we denote by $\Lambda$. The Jacobian $J_{0}(N)$ is simply the quotient $ \mathbb{C}^{g}/ \Lambda$.

Any cusp form $f \in S_{2}(N)$ has a Fourier expansion of the form
\begin{center}
    $f(\tau) = \displaystyle \sum_{n=1}^{\infty} a_{n} e^{2 \pi i \tau}$
\end{center}
where the coefficients $a_{n} \in \mathbb{C}$ are in fact algebraic integers. Let $g_{1}, \ldots, g_{d}$ be the set of Galois conjugates of $f$. The \textit{period lattice} of $f$ is the $\mathbb{Z}$-span of the vectors 
\begin{center}
    $( \langle \gamma_{i}, g_{1} \rangle, \ldots, \langle \gamma_{i}, g_{d} \rangle ) $
\end{center}
where $1 \le i \le 2g$. This is a lattice of rank $2d$ in $ \mathbb{C}^{d}$, and the quotient $A_{f} = \mathbb{C}^{d}/ \Lambda$ is an abelian variety defined over $\mathbb{Q}$. Moreover, there is a natural quotient map 
\begin{center}
    $\varphi_{f} : J_{0}(N) \longrightarrow A_{f}$
\end{center}
which is also defined over $\mathbb{Q}$. Notably, when $d=1$, $A_{f}$ is simply the elliptic curve associated to $f$ and $\varphi_{f}$ is the modular parametrisation. 

For $N \in \{ 3^5, 17^2, 19^2, 29^2, 37^2,41^2\}$, we note that $S_{2}(N)$ contains a positive-rank newform $f$: of dimension $1$ for $N \in \{3^5, 17^2, 19^2, 37^2\}$; of dimension $2$ for $N \in \{29^2, 31^2 41^2\}$; and of dimension $3$ for $N = 7^3$. Let $D$ be a chosen scaled shadow which we assume to be torsion in $J_{0}(N)$ (we use the shadow corresponding to $T_{2}$, except for $7^3$ where we use $T_{3}$). Then, for any upper bound $B(f)$ of $\vert A_f(\Q)_{\tors} \vert$, $\varphi_{f}([B(f) \cdot D]) = 0$ and hence $B(f) \cdot D \in \Lambda_{f}$. 
Showing that this does not occur reduces to a finite computation as we now explain. 

Firstly, bounding $\vert A_f(\Q)_{\tors}\vert$ is elementary, as for any odd prime $l \nmid N$, reduction modulo $l$ induces an injection 
\begin{center}
    $A_f(\Q)_{\text{tors}} \longrightarrow A_f(\mathbb{F}_{l})$
\end{center}
see \cite[Appendix]{katz}. We may therefore take $B(N)$ to be the lowest common multiple of $\vert A_f(\mathbb{F}_{l})\vert$, as $l$ ranges over a finite set of odd primes coprime to $N$. Moreover, the cardinality of $A_f (\mathbb{F}_{l})$ can be determined directly from the coefficients of $f$ as detailed in \cite[Section 3.5]{agashe2005visible}.

Let $D$ be $D' - C$, where $C$ is supported on the cusps and $D'$ is supported away from the cusps. The divisor $\tilde{D} = D' - \deg(D') c_{\infty}$, where $c_{\infty} \in X_0(N)$ is the infinity cusp, differs from $D$ by a divisor which is supported on cusps. By the theorems of Manin \cite{Manin} and  Drinfeld \cite{Drinfeld}, the difference between $D$ and $\tilde{D}$ is torsion, and so $D$ is of infinite order if and only if $\tilde{D}$ is. We therefore prove $\tilde{D}$ is non-torsion, equivalently, $B(f) \tilde{D} \notin \Lambda_f$.
Note that the CM points $\tau_{i} \in \mathbb{H}$ corresponding to the CM points $P_{i}$ in the support of the $D'$ can be explicitly determined in our examples.

Let  $D' = \sum_{j}a_{j}P_{j}$ and construct an integral basis $I$ of the subspace of $S_{2}(N)$ spanned by $f$ and its conjugates. In our calculations, $I$ has at most $3$ elements. Any element $h \in I$ has a Fourier expansion of the form 
\begin{center}
    $h(\tau) = \displaystyle \sum_{n \ge 1} b_{n}^{h} e^{2 \pi i \tau n}$
\end{center}
with $b_{n}^{h} \in \mathbb{Z}$; and we use this to verify that
\begin{center}
$\left( \displaystyle \sum_{j\ge 1} a_{j} \int_{\infty}^{\tau_{j}} 2 \pi i h(z) dz : h \in I \right) \not \in \Lambda_{f}$. 
\end{center}
In particular, from Deligne's proof of the Weil conjectures \cite{deligne1974conjecture}, we deduce that $\vert C b_{n}^{h} \vert \leq d(n)n^{1/2}$, where $d(n)$ denotes the number of positive divisors of $n$ and $C$ is an explicit constant depending on how $h$ is expressed in terms of the basis of normalised eigenforms. We use this to bound the absolute value of the integral of the truncated series 
\begin{center}
 $h_{m} \left( \tau \right) = \displaystyle \sum_{n=1}^{m} b_{n}^{h} e^{2 \pi i \tau n}$
\end{center}
by the geometric series $\displaystyle \sum_{n=1}^{m} 2e^{2n \pi i \tau }$. Summing this geometric series we obtain an estimate for the error, and so we can deduce a choice of $m$ which guarantees convergence to within a chosen tolerance. This allows us to approximate the value of the integral. We compute the period lattice associated to $f$ to within the same tolerance using the \texttt{Magma} command  \texttt{Periods()}, see \cite{cremona1997algorithms} for details of this algorithm. The final step is a simple linear algebra computation which verifies that the point does not lie in the lattice. 

These calculations can often be time consuming due to the need to compute many terms of the Fourier expansions, and Hecke operator calculations become increasingly expensive. The worst case was $41^2$, which took over 12 hours due to needing the first 6600 terms of the expansion.

Using the strategy outlined above, we proved the following. 
\begin{proposition}
  For $N \in \{ 3^5, 7^3, 17^2, 19^2, 29^2, 31^2, 37^2,41^2\}$, $\Cer(N)$ is non-zero.   
\end{proposition}

The online repository contains the \texttt{Magma} code used in our calculations. 

\section{Hecke Components of the Modified Diagonal} \label{heckestructure}
The motivation for studying Ceresa cycles of modular curves comes in part from the connection to triple product $L$-functions. Gross and Kudla \cite{gross1992heights} conjectured that for any three cusp forms with global root number $-1$, the Beilinson-Bloch height of the corresponding isotypic component of the modified diagonal cycle controls the vanishing of the derivative of the $L$-function associated to their product.  We briefly review this here and translate the non-vanishing of the isotypic component of the modified diagonal cycle into a computation using shadow points. 

Throughout this section we  write $\Gamma^{3} \in \CH_{1}(X_{0}(N)^{3})$ for the modified diagonal cycle, with base divisor $\frac{K_{C}}{2g -2}$.

 Following Gross and Kudla \cite[Section 13]{gross1992heights}, we consider the space $P = \CH_1(X_0(N)^3)_\Q^0/\ker^{BB}  $
where $\ker^{BB}$ is the kernel of the Beilinson-Bloch pairing. Hecke correspondences on $X_0(N)$ define symmetric correspondences on $X_0(N)^3$, and for any $(f,g,h) \in S_2(\Gamma_0(N))^{\otimes 3}$ with global root number $-1$, we have a well-defined $f \times g \times h$-isotypic component $\Gamma_{fgh}^3 \in P_{\mathbb{R}}$ of the modified diagonal cycle $ \Gamma^3 \in  \CH_1(X_0(N)^3)$. 

Recall that for any Hecke correspondence  $T$  we have an associated shadow point 
$ \Sh( \Gamma^{3}, T) \in J_{0}(N)(\Q)$.  Similarly we define $\Sh( \Gamma^{3}_{fgh}, T) \in J_0(N) \otimes \mathbb{R} $ to be the shadow point of any lift of $\Gamma_{fgh}^3$ to  
$\CH_1(X_0(N)^3)$. This is well defined by the following, which we thank the anonymous referee for suggesting. 

\begin{proposition}
    The shadow $\Sh(\Gamma^3_{fgh}, T) \in J_0(N)\otimes\mathbb{R}$ is well-defined.
\end{proposition}
\begin{proof}
    We pick two lifts $\Gamma^3_{fgh, 1}$ and $\Gamma^3_{fgh, 2}$ of $\Gamma^3_{fgh}$ to $\CH_1(X_0(N)^3) \otimes \mathbb{R}$. For a Hecke operator $T$, the shadows $\Sh(\Gamma^3_{fgh, 1}, T)$ and $\Sh(\Gamma^3_{fgh, 2}, T)$ differ by $\Sh(\Gamma^3_{fgh, 1}- \Gamma^3_{fgh, 2}, T) = x_0$ by linearity of shadows. The difference is by definition in the kernel of the Beilinson--Bloch height pairing. The shadow can be equivalently computed as the image of $Z = \Gamma^3_{fgh, 1} - \Gamma^3_{fgh, 2}$ under the correspondence $T \times \Delta \in \CH(X_0(N)^4)$ mapping $\CH_1(X_0(N)^3)$ to $\CH^1(X_0(N))$. Since the Beilinson--Bloch height pairing respects correspondences \cite{KunnHigherPic}, 
    $$\langle x_0, x_0 \rangle_{BB} = \langle x_0, (T \times \Delta)_* Z \rangle_{BB} = \langle (T \times \Delta)^\vee_* x_0, Z \rangle_{BB} = 0.$$
    The last equality holds because $Z$ is in the kernel. As the Beilinson--Bloch pairing on curves is equivalent to the N{\'e}ron--Tate height, this shows $x_0$ is torsion and so in particular is zero in $\CH^1(X_0(N)) \otimes \mathbb{R}$.
\end{proof}

By the decomposition of $\Gamma^3$ and the linearity of $\Sh$ we have the following. 
\begin{proposition} Let $T$ be Hecke correspondence over $X_0(N)$. Then $\Sh(\Gamma^3, T)$ decomposes as follows 
\begin{center}
    $\Sh(\Gamma^{3}, T) = \displaystyle \sum_{f,g,h} \Sh( \Gamma^{3}_{fgh}, T). $
\end{center}
\end{proposition}

We show that certain $\Gamma_{fgh}^3$ do not vanish by showing that there exist non-torsion shadow points. We begin with some properties of these shadows under the actions of Hecke operators. A simple consequence of Proposition \ref{breakups} is the following.

\begin{corollary} \label{heckact}
For any Hecke correspondences $\phi_{1}, \phi_{2}, 
\phi_{3}, T$, 
\begin{center}
    $\Sh( (\phi_{1}, \phi_{2}, \phi_{3})_{*}\Gamma^{3}, T) = \phi_{3} \cdot \Sh( \Gamma^{3}, \phi_{2} T \phi_{1}).$
\end{center}
\end{corollary}

\begin{proposition} \label{prop3}
 Let $f,g,h \in \mathcal{S}_{2}(\Gamma_{0}(N))$ be Hecke eigenforms. If $f \ne g$, then $\Sh(\Gamma^3_{fgh}, T) =0 $ for any Hecke correspondence $T$. 
\end{proposition}
\begin{proof}
 For any prime $q$, we have by definition and Corollary~\ref{heckact},
 \begin{center}
     $a_{f}(q)\Sh(\Gamma^3_{fgh}, T) = \Sh( (T_{q},1,1)_{*} \Gamma^3_{fgh}, T) = \Sh(\Gamma^3_{fgh}, TT_q),$
 \end{center}
 where for $q \mid N$, $T_q$ is the Hecke operator $U_q$. By commutativity of the Hecke algebra for $\Gamma_0(N)$, and Corollary~\ref{heckact}.
 \begin{center}
     $\Sh(\Gamma^3_{fgh}, TT_q) = \Sh(\Gamma^3_{fgh}, T_q T) = \Sh( (1, T_{q},1)_{*} \Gamma^3_{fgh}, T) = a_{g}(q) \Sh( \Gamma^3_{fgh}, T).$
 \end{center}
Thus if $\Sh( \Gamma^3_{fgh}, T) \neq 0$, $a_f(q) = a_g(q)$ for all $q$, and so $f = g$.
\end{proof}
\begin{remark}
  We note that this does not force the vanishing of $\Gamma^3_{fgh}$ for $f \neq g$, simply that this does not contribute to shadows.  
\end{remark}

A special case occurs in the case of an Atkin-Lehner operator $w_{d}$ and the diagonal $T = \Delta$ on $X_{0}(N)^{2}$, that is, the trivial correspondence. 
As $w_{d} \otimes w_{d} \otimes w_{d}$ acts trivially on $\Gamma^{3}$, we decompose $\Gamma^{3}$ according to the sign with which $w_{d}$ acts on $\Gamma^{3}$
\begin{center}
    $\Gamma^{3} = \Gamma^{3}_{+++} + \Gamma^{3}_{+--} + \Gamma^{3}_{-+-} + \Gamma^{3}_{--+}$.
\end{center}

\begin{proposition}
If $\Sh(\Gamma^{3}, w_{d}) \ne 0$, then  $\Sh( \Gamma^{3}_{+++}, \Delta) \ne 0$ and $\Sh( \Gamma^{3}_{--+}, \Delta) \ne 0$. Moreover, all four Atkin-Lehner components of $\Gamma^3$ are non-zero. Conversely, if $\Sh(\Gamma^{3}, w_{d}) =0$ then $ \Sh( \Gamma^{3}_{+++}, \Delta) =0 $ and $ \Sh( \Gamma^{3}_{--+}, \Delta) =0$.
\end{proposition}
\begin{proof}
Using the above decomposition we obtain 
\begin{align*}
    & \Sh( ( w_{d} \otimes 1 \otimes 1)_{*}\Gamma^{3}, \Delta) =  \Sh( \Gamma^{3}_{+++}, \Delta) + \Sh( \Gamma^{3}_{+--}, \Delta) - \Sh( \Gamma^{3}_{-+-}, \Delta) - \Sh( \Gamma^{3}_{--+}, \Delta); \\
    & \Sh( ( 1 \otimes  w_{d}  \otimes 1)_{*}\Gamma^{3}, \Delta) =  \Sh( \Gamma^{3}_{+++}, \Delta) -  \Sh( \Gamma^{3}_{+--}, \Delta) +  \Sh( \Gamma^{3}_{-+-}, \Delta) - \Sh( \Gamma^{3}_{--+}, \Delta). 
    \end{align*}
 From Corollary \ref{heckact}, we deduce that both of the above divisors are $\Sh(\Gamma^{3}, w_{d})$ and thus 
 \begin{center}
     $\Sh( \Gamma^{3}, w_{d}) = \Sh( \Gamma^{3}_{+++}, \Delta) - \Sh( \Gamma^{3}_{--+}, \Delta)$
 \end{center}
 Moreover, we have
    \begin{align*}
   & \Sh( ( 1  \otimes 1 \otimes w_{d} )_{*}\Gamma^{3}, \Delta) =  \Sh( \Gamma^{3}_{+++}, \Delta) -  \Sh( \Gamma^{3}_{+--}, \Delta) -   \Sh( \Gamma^{3}_{-+-}, \Delta) +  \Sh( \Gamma^{3}_{--+}, \Delta);
    \\
    & \Sh( \Gamma^{3}, \Delta) =  \Sh( \Gamma^{3}_{+++}, \Delta) +  \Sh( \Gamma^{3}_{+--}, \Delta) +  \Sh( \Gamma^{3}_{-+-}, \Delta) +  \Sh( \Gamma^{3}_{--+}, \Delta);
    \end{align*}
and since shadow points are invariant under $w_{d}$ and $\Sh( \Gamma^{3}, \Delta) =0$, 
\begin{center}
    $ \Sh( \Gamma^{3}_{+++}, \Delta)  = -  \Sh( \Gamma^{3}_{--+}, \Delta);$
\end{center}
and the claim follows for $\Gamma^3_{+++}$ and $\Gamma^3_{--+}$. The claim for the other two components follows by the invariance of $\Gamma^3$ under permutation of the factors.
\end{proof}

One motivation for studying this decomposition is the following conjecture of Gross and Kudla. For any three weight $2$ cusps forms $f,g,h$ of level $\Gamma_{0}(N)$, with $N$ squarefree, consider $F = f \times g \times h $ on $\mathbb{H}^{3}$. The associated $L$-function has been classically studied by Garrett \cite{garrett1987decomposition} and Piatetski-Shapiro and Rallis \cite{piatetski1987rankin}, in particular it is centred at $s=2$. We additionally assume that $F$ is such that $L(F,s)$ has global root number $-1$.  Associated to this setting, there is a canonically defined Shimura curve $X$, and as before we write $\Gamma^{3}$ for the modified diagonal cycle attached to $X$. Gross and Kudla \cite{gross1992heights} conjecture the following. 
\begin{conjecture} With notation as above
\begin{center}
 $L'(F,2) = \Omega(F) \langle \Gamma^{3}_{F}, \Gamma^{3}_{F} \rangle_{BB} $   
 \end{center}
 where $\Omega(F)$ is an explicit,non-zero constant, $\Gamma^{3}_{F}$ is the $F$-isotypic component of $\Gamma^{3}$ and $\langle, \rangle_{BB}$ is the Bloch-Beilinson height pairing.
\end{conjecture}

 There has been some recent progress towards proving this conjecture, see \cite{yuanzhangzhang}. We provide some examples to support this conjecture where we prove the non-vanishing of $\Gamma^3_F$ and $L'(F, 2) \neq 0$. For instance, suppose that $X = X_{0}(p)$ is bielliptic with non-vanishing Ceresa cycle, that is, $ p \in \{ 43, 53, 61, 79, 83, 101, 131 \}$ (see \cite{bars1999bielliptic}). For all such levels there is a unique $f_{p} \in \mathcal{S}_{2}(p)$ with Atkin-Lehner sign $1$ and in the following two results we exploit this to deduce that $ \Sh( \Gamma^{3}_{ggf_{p}}, \Delta) \ne 0$ for all $g \in \mathcal{S}_{2}(p)$, $ g \ne f_{p}$.

\begin{proposition} \label{prop2}
 Let $X = X_{0}(p)$ with $p \in \{ 43, 53, 61, 79, 83, 101, 131 \}$. Then $\Sh( \Gamma^{3}, w_{p}) =0$. In particular, $ \Gamma^3_{f_{p}f_{p}f_{p}}$ does not contribute to shadow points.
\end{proposition}
\begin{proof}
By the previous proposition, it suffices to show $ \Sh(\Gamma^{3},w_{p}) =0$. This follows from our definition of $ \Sh(\Gamma^{3}, w_{p})$ and the fact that the canonical divisor is the sum of the fixed points of $ w_{p}$, since the curve is bi-elliptic. 
\end{proof}
In fact, $\Gamma^3_{f_p f_p f_p} = 0$ as it is the pullback of the diagonal cycle on the elliptic curve \cite[Corollary 4.7]{GrossSchoen}, but for our purposes the statement of the proposition suffices. 

\begin{proposition} \label{prop1}
    Let $X = X_{0}(p)$ and $f_{p}$ be as above. Then $\Sh(\Gamma^{3}_{ggf_{p}}, \Delta)  \ne 0$ for all cusps forms $g$ with Atkin-Lehner sign $-1$.
\end{proposition}
\begin{proof}
 By Propositions \ref{prop3} and \ref{prop2}, and the fact that $\Sh(\Gamma^{3}, \Delta) = 0$, we have 
\begin{center}
    $ \displaystyle \sum_{g} \Sh(\Gamma^{3}_{ggf_{p}}, \Delta) =0$
\end{center}
where the sum is taken over a basis for the space of weight 2 cusps forms with Atkin-Lehner sign $-1$. Moreover, for all primes $\ell \ne p$
\begin{center}
    $ \Sh( \Gamma^{3}, T_{\ell}) = \displaystyle \sum_{g} a_{g,\ell}\Sh( \Gamma_{ggf_{p}}^{3}, \Delta)  $
\end{center}
where $a_{g,\ell}$ is the eigenvalue of $T_{\ell}$ acting on $g$. Let $d$ be the dimension of the $-1$ Atkin-Lehner eigenspace. Taking $d-1$ different primes $\ell$, with at least one $\Sh(\Gamma^{3}, T_{\ell}) \ne 0$ (such a prime exists from our proof of Theorem \ref{THM1}, and can in fact be taken to be $2$), gives a system of $d$ equations in $d$ variables.

As the Mordell--Weil group of $X_0^+(p)$ is $\mathbb{Z}$, these relations can be described as $Ax = b$, where $A$ and $x$ have entries in the Galois closure of $\mathbb{Q}(g)$, and $b$ has integer entries. The columns of $A$ are of the form $\begin{pmatrix}1 & \sigma(a_{g, \ell_1}) & \ldots & \sigma(a_{g, \ell_{d - 1}}) \end{pmatrix}$ as $\sigma$ ranges over the real embeddings of $\mathbb{Q}(g)$. As the embeddings of a number field are linearly independent and the $a_{g, \ell_1}$ can be assumed to form a basis for the field, the columns of $A$ are thus linearly independent. In particular, $A$ is invertible so a unique solution exists. The vector $b$ is non-zero as at least one shadow is, and so $x$ is non-zero. As the Galois group of $\mathbb{Q}$ permutes the columns of $A$ transitively and fixes $b$, it must permute the rows of $x$ transitively. Since at least one entry of $x$ is non-zero, they all are.
\end{proof}

We now turn to the $L$-function side of the conjecture.
\begin{proposition} 
 Let  $ p \in \{ 43, 53, 61, 79, 83, 101, 131 \}$. For all $g \in \mathcal{S}_{2}(p)$ with Atkin-Lehner sign $-1$, $L'(g \otimes g \otimes f_{p}, 2) \ne 0 $.
 \end{proposition}
\begin{proof}
  We make use of the factorisation $L(g \otimes g \otimes f_p) = L(f_p \otimes \mathrm{Sym}^2 g, s) L(f_p, s - 1)$ to reduce calculating $L'(g \otimes g \otimes f_p, 2)$ to $L'(f_p, 1)$ and $L(f_p \otimes \mathrm{Sym}^2 g, 2)$. These can be computed using the computer algebra package $\mathtt{Magma}$, after correcting the local $L$-factor at $p$. When forming tensor products of L-series, \texttt{Magma} does not check if bad primes for both representations become good for the tensor product, see example H139E44 in the Magma Handbook \cite{magma}.
\end{proof}

In theory, such calculations can be carried out for any $X_{0}(N)$.
\begin{eg}
 Let $X = X_{0}(67)$. The $+$ space is spanned by two conjugate cusp forms $h_{1}, h_{2}$ defined over $\Q(\sqrt{5})$; whilst the $-$ space is spanned by a rational cusp form $f$ and a pair of Galois conjugate forms $g_{1}, g_{2}$ defined over $ \Q( \sqrt{5})$. 
 
 We compute models of $X_{0}(67)$ and $ X_{0}^{+}(67)$, along with an isomorphism $J_{0}^{+}(67)(\Q) \cong \mathbb{Z}^{2}$. Shadow points can be determined explicitly using our models and various facts about quadratic points proved in \cite{box2021quadratic}. We push-forward the divisors onto $J_{0}^{+}(67)(\mathbb{Q})$ to find: 
 \begin{center}
  $3\Sh(T_{2}) = 24(1,3), 3 \Sh(T_{3}) = 24(-10, 3), 3 \Sh(T_{7}) = 24( -1, -7), 3 \Sh(w_{67}) = 48(1, -1),$   
 \end{center}
 with respect to a fixed basis of $J_{0}^{+}(67) \cong \mathbb{Z}^{2}$.
Along with $\Sh( \Delta) =0$, and the images of these 5 points under $T_2$, we obtain a system of 10 equations in the 10 variables $ a_{s,t} = \Sh( \Gamma^{3}_{sst}, \Delta)$ where $s \in \{ h_{1},h_{2}, f, g_{1}, g_{2} \} $ and $t \in \{ h_{1}, h_{2} \}$.
Explicitly solving this system, we find that  $\Sh(\Gamma^{3}_{h_{1}h_{1}h_{1}}, \Delta) = \Sh(\Gamma^{3}_{h_{2}h_{2}h_{2}}, \Delta) = 0 $,  which is consistent with the vanishing of the respective triple product $L$-function. All other Hecke components are non-zero, as are the corresponding derivatives of the associated $L$-function. 
\end{eg}
\bibliographystyle{abbrv}
\bibliography{ref}
\end{document}